\documentclass[a4paper,UKenglish,cleveref, autoref, thm-restate]{lipics-v2021}

\title{Classification of Covering Spaces and Canonical Change of Basepoint} 

\author{Jelle Wemmenhove\footnote{Corresponding author}}{Department of Mathematics and Computer Science, Eindhoven University of Technology, The~Netherlands}{a.j.wemmenhove@tue.nl}{https://orcid.org/0000-0002-4911-2922}{}
\author{Cosmin Manea}{Department of Mathematics and Computer Science, Eindhoven University of Technology, The~Netherlands}{}{https://orcid.org/0009-0004-1940-1215}{}

\author{Jim Portegies}{Department of Mathematics and Computer Science, Eindhoven University of Technology, The~Netherlands}{}{https://orcid.org/0000-0002-2103-7334}{}

\authorrunning{J. Wemmenhove, C. Manea, and J. Portegies} 

\Copyright{Jelle Wemmenhove, Cosmin Manea, and Jim Portegies} 



\relatedversion{} 

\supplement{}
\supplementdetails[
linktext={https://gitlab.tue.nl/computer-verified-\\proofs/covering-spaces},
cite=project-repo,
subcategory={Source Code}, 
swhid={swh:1:dir:bea0c0af55e3ec9869679f3a5611cc5154a9ddbf;origin=https://gitlab.tue.nl/computer-verified-proofs/covering-spaces;visit=swh:1:snp:6b1d9a68d2f10e958534db29214f30ddd1dd9db3;anchor=swh:1:rev:cfc827a2b07f8cb93b412bd6e551d31fae044fb9}]
{Software}{https://gitlab.tue.nl/computer-verified-proofs/covering-spaces} 

\nolinenumbers 

\EventEditors{Delia Kesner, Eduardo Hermo Reyes, and Benno van den Berg}
\EventNoEds{3}
\EventLongTitle{29th International Conference on Types for Proofs and Programs (TYPES 2023)}
\EventShortTitle{TYPES 2023}
\EventAcronym{TYPES}
\EventYear{2023}
\EventDate{June 12-16, 2023}
\EventLocation{ETSInf, Universitat Polit\`{e}cnica de Val\`{e}ncia, Spain}
\EventLogo{}
\SeriesVolume{303}
\ArticleNo{1}

\usepackage{xspace}
\usepackage{mathtools}
\newcommand{\jdeq}{\equiv}      

\newcommand{\defeq}{\vcentcolon\equiv}  

\makeatletter
\def\prd#1{\@ifnextchar\bgroup{\prd@parens{#1}}{%
    \@ifnextchar\sm{\prd@parens{#1}\@eatsm}{%
    \@ifnextchar\prd{\prd@parens{#1}\@eatprd}{%
    \@ifnextchar\;{\prd@parens{#1}\@eatsemicolonspace}{%
    \@ifnextchar\\{\prd@parens{#1}\@eatlinebreak}{%
    \@ifnextchar\narrowbreak{\prd@parens{#1}\@eatnarrowbreak}{%
      \prd@noparens{#1}}}}}}}}
\def\prd@parens#1{\@ifnextchar\bgroup%
  {\mathchoice{\@dprd{#1}}{\@tprd{#1}}{\@tprd{#1}}{\@tprd{#1}}\prd@parens}%
  {\@ifnextchar\sm%
    {\mathchoice{\@dprd{#1}}{\@tprd{#1}}{\@tprd{#1}}{\@tprd{#1}}\@eatsm}%
    {\mathchoice{\@dprd{#1}}{\@tprd{#1}}{\@tprd{#1}}{\@tprd{#1}}}}}
\def\@eatsm\sm{\sm@parens}
\def\prd@noparens#1{\mathchoice{\@dprd@noparens{#1}}{\@tprd{#1}}{\@tprd{#1}}{\@tprd{#1}}}
\def\lprd#1{\@ifnextchar\bgroup{\@lprd{#1}\lprd}{\@@lprd{#1}}}
\def\@lprd#1{\mathchoice{{\textstyle\prod}}{\prod}{\prod}{\prod}({\textstyle #1})\;}
\def\@@lprd#1{\mathchoice{{\textstyle\prod}}{\prod}{\prod}{\prod}({\textstyle #1}),\ }
\def\tprd#1{\@tprd{#1}\@ifnextchar\bgroup{\tprd}{}}
\def\@tprd#1{\mathchoice{{\textstyle\prod_{(#1)}}}{\prod_{(#1)}}{\prod_{(#1)}}{\prod_{(#1)}}}
\def\dprd#1{\@dprd{#1}\@ifnextchar\bgroup{\dprd}{}}
\def\@dprd#1{\prod_{(#1)}\,}
\def\@dprd@noparens#1{\prod_{#1}\,}

\def\@eatnarrowbreak\narrowbreak{%
  \@ifnextchar\prd{\narrowbreak\@eatprd}{%
    \@ifnextchar\sm{\narrowbreak\@eatsm}{%
      \narrowbreak}}}
\def\@eatlinebreak\\{%
  \@ifnextchar\prd{\\\@eatprd}{%
    \@ifnextchar\sm{\\\@eatsm}{%
      \\}}}
\def\@eatsemicolonspace\;{%
  \@ifnextchar\prd{\;\@eatprd}{%
    \@ifnextchar\sm{\;\@eatsm}{%
      \;}}}

\def\lam#1{{\lambda}\@lamarg#1:\@endlamarg\@ifnextchar\bgroup{.\,\lam}{.\,}}
\def\@lamarg#1:#2\@endlamarg{\if\relax\detokenize{#2}\relax #1\else\@lamvar{\@lameatcolon#2},#1\@endlamvar\fi}
\def\@lamvar#1,#2\@endlamvar{(#2\,{:}\,#1)}
\def\@lameatcolon#1:{#1}

\def\lamu#1{{\lambda}\@lamuarg#1:\@endlamuarg\@ifnextchar\bgroup{.\,\lamu}{.\,}}
\def\@lamuarg#1:#2\@endlamuarg{#1}

\def\fall#1{\forall (#1)\@ifnextchar\bgroup{.\,\fall}{.\,}}

\def\exis#1{\exists (#1)\@ifnextchar\bgroup{.\,\exis}{.\,}}

\def\sm#1{\@ifnextchar\bgroup{\sm@parens{#1}}{%
    \@ifnextchar\prd{\sm@parens{#1}\@eatprd}{%
    \@ifnextchar\sm{\sm@parens{#1}\@eatsm}{%
    \@ifnextchar\;{\sm@parens{#1}\@eatsemicolonspace}{%
    \@ifnextchar\\{\sm@parens{#1}\@eatlinebreak}{%
    \@ifnextchar\narrowbreak{\sm@parens{#1}\@eatnarrowbreak}{%
        \sm@noparens{#1}}}}}}}}
\def\sm@parens#1{\@ifnextchar\bgroup%
  {\mathchoice{\@dsm{#1}}{\@tsm{#1}}{\@tsm{#1}}{\@tsm{#1}}\sm@parens}%
  {\@ifnextchar\prd%
    {\mathchoice{\@dsm{#1}}{\@tsm{#1}}{\@tsm{#1}}{\@tsm{#1}}\@eatprd}%
    {\mathchoice{\@dsm{#1}}{\@tsm{#1}}{\@tsm{#1}}{\@tsm{#1}}}}}
\def\@eatprd\prd{\prd@parens}
\def\sm@noparens#1{\mathchoice{\@dsm@noparens{#1}}{\@tsm{#1}}{\@tsm{#1}}{\@tsm{#1}}}
\def\lsm#1{\@ifnextchar\bgroup{\@lsm{#1}\lsm}{\@@lsm{#1}}}
\def\@lsm#1{\mathchoice{{\textstyle\sum}}{\sum}{\sum}{\sum}({\textstyle #1})\;}
\def\@@lsm#1{\mathchoice{{\textstyle\sum}}{\sum}{\sum}{\sum}({\textstyle #1}),\ }
\def\tsm#1{\@tsm{#1}\@ifnextchar\bgroup{\tsm}{}}
\def\@tsm#1{\mathchoice{{\textstyle\sum_{(#1)}}}{\sum_{(#1)}}{\sum_{(#1)}}{\sum_{(#1)}}}
\def\dsm#1{\@dsm{#1}\@ifnextchar\bgroup{\dsm}{}}
\def\@dsm#1{\sum_{(#1)}\,}
\def\@dsm@noparens#1{\sum_{#1}\,}

\def\wtype#1{\@ifnextchar\bgroup%
  {\mathchoice{\@twtype{#1}}{\@twtype{#1}}{\@twtype{#1}}{\@twtype{#1}}\wtype}%
  {\mathchoice{\@twtype{#1}}{\@twtype{#1}}{\@twtype{#1}}{\@twtype{#1}}}}
\def\lwtype#1{\@ifnextchar\bgroup{\@lwtype{#1}\lwtype}{\@@lwtype{#1}}}
\def\@lwtype#1{\mathchoice{{\textstyle\mathsf{W}}}{\mathsf{W}}{\mathsf{W}}{\mathsf{W}}({\textstyle #1})\;}
\def\@@lwtype#1{\mathchoice{{\textstyle\mathsf{W}}}{\mathsf{W}}{\mathsf{W}}{\mathsf{W}}({\textstyle #1}),\ }
\def\twtype#1{\@twtype{#1}\@ifnextchar\bgroup{\twtype}{}}
\def\@twtype#1{\mathchoice{{\textstyle\mathsf{W}_{(#1)}}}{\mathsf{W}_{(#1)}}{\mathsf{W}_{(#1)}}{\mathsf{W}_{(#1)}}}
\def\dwtype#1{\@dwtype{#1}\@ifnextchar\bgroup{\dwtype}{}}
\def\@dwtype#1{\mathsf{W}_{(#1)}\,}

\def\wtypeh#1{\@ifnextchar\bgroup%
  {\mathchoice{\@lwtypeh{#1}}{\@twtypeh{#1}}{\@twtypeh{#1}}{\@twtypeh{#1}}\wtypeh}%
  {\mathchoice{\@@lwtypeh{#1}}{\@twtypeh{#1}}{\@twtypeh{#1}}{\@twtypeh{#1}}}}
\def\lwtypeh#1{\@ifnextchar\bgroup{\@lwtypeh{#1}\lwtypeh}{\@@lwtypeh{#1}}}
\def\@lwtypeh#1{\mathchoice{{\textstyle\mathsf{W}^h}}{\mathsf{W}^h}{\mathsf{W}^h}{\mathsf{W}^h}({\textstyle #1})\;}
\def\@@lwtypeh#1{\mathchoice{{\textstyle\mathsf{W}^h}}{\mathsf{W}^h}{\mathsf{W}^h}{\mathsf{W}^h}({\textstyle #1}),\ }
\def\twtypeh#1{\@twtypeh{#1}\@ifnextchar\bgroup{\twtypeh}{}}
\def\@twtypeh#1{\mathchoice{{\textstyle\mathsf{W}^h_{(#1)}}}{\mathsf{W}^h_{(#1)}}{\mathsf{W}^h_{(#1)}}{\mathsf{W}^h_{(#1)}}}
\def\dwtypeh#1{\@dwtypeh{#1}\@ifnextchar\bgroup{\dwtypeh}{}}
\def\@dwtypeh#1{\mathsf{W}^h_{(#1)}\,}

\makeatother

\newcommand{\proj}[1]{\ensuremath{\mathsf{pr}_{#1}}\xspace}
\newcommand{\fst}{\ensuremath{\proj1}\xspace}

\newcommand{\ttype}{\ensuremath{\mathsf{Type}}} 

\newcommand{\refl}[1]{\ensuremath{\mathsf{refl}_{#1}}\xspace}

\newcommand{\ct}{%
  \mathchoice{\mathbin{\raisebox{0.5ex}{$\displaystyle\centerdot$}}}%
             {\mathbin{\raisebox{0.5ex}{$\centerdot$}}}%
             {\mathbin{\raisebox{0.25ex}{$\scriptstyle\,\centerdot\,$}}}%
             {\mathbin{\raisebox{0.1ex}{$\scriptscriptstyle\,\centerdot\,$}}}
}

\newcommand{\transfib}[3]{\ensuremath{\mathsf{transport}^{#1}(#2,#3)\xspace}}

\let\apfunc\mapfunc

\let\apdfunc\mapdepfunc

\newcommand{\idfunc}[1][]{\ensuremath{\mathsf{id}_{#1}}\xspace}

\newcommand{\eqvsym}{\simeq}    

\let\type\UU

\newcommand{\set}{\ensuremath{\mathsf{Set}}\xspace} 

\newcommand{\prop}{\ensuremath{\mathsf{Prop}}\xspace}

\newcommand{\idtoeqv}{\ensuremath{\mathsf{idtoeqv}}\xspace}

\newcommand{\trunc}[2]{\mathopen{}\left\Vert #2\right\Vert_{#1}\mathclose{}}

\newcommand{\Trunc}[2]{\Bigl\Vert #2\Bigr\Vert_{#1}}

\newcommand{\tproj}[3][]{\mathopen{}\left|#3\right|_{#2}^{#1}\mathclose{}}
\newcommand{\tprojf}[2][]{|\blank|_{#2}^{#1}}

\newcommand{\brck}[1]{\trunc{}{#1}}

\newcommand{\Brck}[1]{\Trunc{}{#1}}
\newcommand{\bproj}[1]{\tproj{}{#1}}

\newcommand{\base}{\ensuremath{\mathsf{base}}\xspace}

\newcommand{\blank}{\mathord{\hspace{1pt}\text{--}\hspace{1pt}}}

\newcommand{\happly}{\mathsf{happly}}

\newcommand{\rrefl}[2]{\ensuremath{\mathsf{refl}_{#1}^{#2}}\xspace}
\newcommand{\pto}[0]{\ensuremath{\,\,\cdotp \to}\xspace}
\usepackage{tikz-cd}

\begin{document}

\maketitle

\begin{abstract}
Using the language of homotopy type theory (HoTT), we 1) prove a synthetic version of the classification theorem for covering spaces, and 2) explore the existence of canonical change-of-basepoint isomorphisms between homotopy groups.
There is some freedom in choosing how to translate concepts from classical algebraic topology into HoTT. The final translations we ended up with are easier to work with than the ones we started with. We discuss some earlier attempts to shed light on this translation process. The proofs are mechanized using the \textsc{Coq} proof assistant and closely follow classical treatments like those by Hatcher \cite{hatcherAT}.
\end{abstract}

\section{Introduction}

\noindent Homotopy type theory (HoTT) is a variant of Martin-L\"of type theory (MLTT) that can be used as a synthetic language for developing the theory of algebraic topology, specifically the subfield of homotopy theory. This means that the types in MLTT are given topological interpretations. The main example is the identity type $a =_X b$ which in HoTT is interpreted as the type of paths from $a$ to $b$ in a space $X$\,. Since identity types are primitive objects in MLTT, one obtains a low-level encoding of the mathematical theory: definitions are simple and you can quickly go on to prove interesting theorems.

Many classical results from algebraic topology and homotopy theory have already been developed synthetically within HoTT, e.g. computations of homotopy groups of spheres (Brunerie and Licata \cite{brunerie-thesis,homotopy-spheres}, and Ljungstr\"om and M\"ortberg~\cite{computation_brunerie_nr}), the Blakers-Massey theorem (Hou, Finster, Licata, and Lumsdaine \cite{blakers-massey}), and Van Kampen's theorem (Hou and Shulman~\cite{seifert-vankampen}). Not every result, however, can be translated into homotopy type theory directly: to prove a synthetic version of Brouwer's fixed point theorem, for example, requires extending the base theory to so-called \emph{real-cohesive} homotopy type theory \cite{brouwers-fixed-point}. In this article, we build on the development of covering spaces in HoTT by Hou and Harper \cite{favonia-covspaces}, and Buchholtz and Hou's work on cellular cohomology \cite{favonia-cellcohom}.

We prove a synthetic version of the \emph{classification of covering spaces} and synthetically explore the existence of \emph{canonical change-of-basepoint isomorphisms} between homotopy groups. Although these topics seem quite disparate, the motivation to develop these results in HoTT was unitary: we had little experience in using HoTT as a synthetic language, so to increase our knowledge we set out to prove an exercise from Hatcher's \textit{Algebraic Topology} \cite{hatcherAT} (Exercise 3.3.11) in HoTT. The solution to this exercise required the development of the two topics presented here. The results have been mechanized\footnote{Code repository: \url{https://gitlab.tue.nl/computer-verified-proofs/covering-spaces}. The content of this article corresponds to version 0.3.} in the \textsc{Coq} proof assistant using the \textsc{Coq-HoTT} library \cite{coq-hott}. We were able to closely mirror the classical arguments used by Hatcher, making extensive use of core HoTT concepts like `transport' and `truncations'.
With the aim of serving as an entry point for other newcomers to HoTT, we tried to keep our proofs explicit.

Our low-level approach might suggest that HoTT is merely a formal language for re-expressing existing concepts and proofs from algebraic topology. Yet for experts, HoTT does provide new ways to think about their subjects and allows them to express new concepts in high-level conceptual arguments, such as the development of so-called `higher groups' by Buchholtz, Van Doorn, and Rijke \cite{higher-groups}.

The synthetic version of the \emph{classification of covering spaces} is shown in Section \ref{sec:classification}, as well as intermediate results like the \emph{lifting criterion}. We use Hou and Harper's definition of covering spaces \cite{favonia-covspaces} and proof techniques like `extension by weak constancy'~\cite{kraus2015weakextension} to stay close to the classical treatment by Hatcher \cite{hatcherAT}. In contrast, Buchholtz et al. prove the Galois correspondence of covering spaces from a more abstract perspective using their theory of higher groups \cite[Thm.\,7]{higher-groups}. Finding the `right' translations of classical statement into HoTT was an iterative process: after having proven a specific formulation of some theorem, we would realize that the proof could strongly be simplified if we used a different, yet equivalent, formulation. To shed some light on the translation process, Section \ref{sec:classification} also discusses some of our earlier attempts.

In Section \ref{sec:change_of_basepoint}, we prove conditions for the existence of \emph{canonical change-of-basepoint isomorphisms} between homotopy groups by looking at the triviality of the $\pi_1$-action on these groups. We were interested in the existence of such isomorphisms because we needed them to define the degree of non-pointed maps between spheres. Hou \cite{favonia-thesis} defines the degree of such maps in a different way: on the level of sets, non-pointed maps between spheres are equivalent to pointed ones, and for these it is easy to define a degree. Although this approach seems quite different from ours, we prove a classical result relating basepoint-preserving and free homotopy classes of maps, which illustrates that Hou's approach also relies on the triviality of the $\pi_1$-actions for spheres.

Let us here also quickly mention a third way to turn non-pointed maps between spheres into pointed ones. As part of recent work, Buchholtz et al.\footnote{See Meyers's talk \emph{The tangent bundles of spheres}, available at \url{https://www.youtube.com/watch?v=9T9B9XBjVpk} (timestamp 9:30 minutes).} constructed a family of sphere reparameterizations that can map any point on the sphere to the basepoint. As such any non-pointed map between spheres can be transformed into a pointed one by post-composition with the appropriate reparameterization.

Before moving on to the main results in Sections \ref{sec:classification} and \ref{sec:change_of_basepoint}, we recall the homotopical interpretation of types and discuss some techniques used throughout this article. 

\section{Background} \label{sec:background}

\subsection{Topological interpretation}

We recall the basic notions from homotopy type theory necessary for understanding the results in this article. For a full explanation of homotopy type theory, we refer to the HoTT book \cite{hottbook}. It starts with a friendly introduction to type theory and its homotopical interpretation, discusses the differences with MLTT, such as higher inductive types and the univalence axiom, and provides notes on the historical development of these ideas.
In this section, we do highlight some nuances that might be missed on an initial reading of the book.

\subsubsection{Spaces, points, paths, and higher paths}
The key to reading type theoretical statements as topological statements is the \emph{identification-as-paths} interpretation: a type $X$ is interpreted as a \emph{space}, terms $x : X$ are interpreted as \emph{points} in this space, and identity types $a =_X b$ are interpreted as the \emph{type of paths} from point~$a$ to $b$\,. Viewing the identity type $a =_X b$ as a space in itself, we obtain a type $p =_{a =_X b} q$ of \emph{paths-between-paths} with $p,q : a =_X b$\,. These paths-between-paths are called \emph{homotopies} in topology. Continuing this construction we get an infinite tower of higher-and-higher paths, endowing every type $X$ with the structure of an $\infty$-groupoid.

\begin{remark*}
Coming from a math background, it is easy to mistake \textit{``a type $X$ is interpreted as a space''} as meaning that $X$ is a \emph{topological} space. This is incorrect! It is meant that $X$ is a space precisely in the sense that it has an $\infty$-groupoid structure. $\infty$-groupoids turn out to be the right data structure for describing the (higher order) path-structure of genuine topological spaces: from a topological space $X'$ one can construct the fundamental $\infty$-groupoid $\Pi_\infty(X')$ and this $\infty$-groupoid is enough to prove many theorems in algebraic topology, specifically the theorems belonging to the subfield of homotopy theory. Words like \emph{path} and \emph{homotopy} are reused for elements of $\infty$-groupoids both to prevent the need for a new vocabulary and to more easily tap into our spatial intuition.
For an excellent discussion on the differences between $\infty$-groupoids and topological spaces, see the introduction of Shulman's work \cite{brouwers-fixed-point}.
\end{remark*}

Even though identity types $a =_X b$ are interpreted as paths, they are also still thought of as equalities. In calculations, it feels more natural to read a chain like $a = \ldots = z$ as a list of equalities than as a composition of paths. The usual rules for equalities also still apply: for example, an equality $p : a =_X b$ implies an equality $\apfunc f (p) : f(a) =_Y f(b)$\,, where $f  :X \to Y$\,. In HoTT, this statement is interpreted as saying that all functions are \emph{continuous}, in the sense that they preserve paths between points.

\subsubsection{Type families, transport, and dependent paths}

A type family $P : X \to \ttype$ is interpreted as a collection of spaces that `lie over' a \emph{base space}~$X$\,. the space $P(x)$ is also referred to as the \emph{fiber} over $x : X$\,.
The sigma type $\sum_{x : X} P(x)$ is called the \emph{total space} of $P$ and the first projection $\fst : (\sum_{x : X} P(x)) \to X$ is a special map called a \emph{fibration}. These maps play an important role in classical homotopy theory, but in HoTT it is easier to work with type families directly.

Given a path $p : a =_X b$\,, points $u : P(a)$ in the fiber over $a$ can be `transported' along $p$ to the fiber $P(b)$\,, the result is denoted as $\transfib P p u : P(b)$\,. Transporting along paths behaves as expected: transport along a constant path $\refl x : x =_X x$ leaves the points in $P(x)$ unchanged, $\transfib P {\refl x} u \jdeq u$\,, transport along a composite path $p \ct q : a =_X b =_X c$ is the same as first doing transport along $p$ and then along $q$\,,
\[
    \transfib P {p \ct q} u =_{P(c)} \transfib P q {\transfib P p u}~,
\]
and transport along the reverse path $p^{-1} : b =_X a$ is the inverse of transport along $p$\,. Chapter 2 of the HoTT book \cite{hottbook} contains a number of useful characterizations of transport in different kinds of type families. Of special interest to us are the formulas for transport in loop spaces, e.g. transport of a loop $q : a =_X a$ along $p : a =_X b$ equals conjugation with $p$\,,
\begin{equation}
    \transfib{x \mapsto x =_X x} p q =_{(b =_X b)} p^{-1} \ct q \ct p~. \label{eq:transp_conj}
\end{equation}

Transport also provides a convenient way to reason about paths in the total space. In general, a path $\widetilde{p} : (a,u) = (b,v)$ in the total space is equivalent to the combination of a path $p : a =_X b$ in the base space and a path $w_p : (\transfib P p u =_{P(b)} v)$ in the fiber $P(b)$\,, see \cite[Thm.\,2.7.2]{hottbook}. A path of the form $(\transfib P p u =_{P(b)} v)$ is called a \emph{dependent} path; it is interpreted as a path from $u$ to $v$ that lies \emph{over} the path $p$\,. For example, the dependent path (\ref{eq:transp_conj}) is interpreted as a path (i.e. a homotopy) from loop $q : a =_X a$ to loop  $p^{-1} \ct q \ct p : b =_X b$ that lies over $p$\,. The interpretation of dependent paths as lying over paths in the base space is justified by the equivalence with paths in the total space. In fact, to specify the paths that lie over some path $p$\,, it is easier to use dependent paths than paths in the total space. This is because the equality $\apfunc \fst (\widetilde p) = p$ which picks out the paths $\widetilde{p} : (a,u) = (b,v)$ that lie over $p$ is propositional and not judgemental.

Finally, we have the following result which says that transport commutes with operations on the fibers of the type family.
\begin{lemma} \label{lem:comm_with_transport}
    Let $P, Q : X \to \ttype$ be two type families and let $f_{(\blank)}:\prod_x (P(x) \to Q(x))$ be a family of maps, then for all paths $p : a =_X b$ and points $u : P(a)$ it holds that
    \[
        \transfib Q p {f_a (u)} = f_b (\transfib P p u)~.
    \]
\end{lemma}

\begin{proof}
    Let $p : a =_X b$\,, then there exists a dependent equality between $f_a$ and $f_b$ over~$p$\,, 
    \[
        \apdfunc{f_{(\blank)}}(p) : \transfib {x \mapsto (P(x) \to Q(x))} p {f_a} = f_b~,
    \]
    and by Lemma 2.9.6 in the HoTT book \cite{hottbook}, this is equivalent to what we need to show.
\end{proof}

\subsubsection{Truncations and path-connectedness}

In classical homotopy theory, the infinite structure of $\infty$-groupoids can be `truncated' to make them easier to study. Such truncation operators are also available in HoTT, but here their main use is as modalities that increase the logical expressiveness of the theory.

For example, propositional truncation is needed to accurately capture the concept of path-connected spaces in homotopy type theory. A space $X$ is called path-connected (or just \emph{connected} in HoTT, as there is no analogy to the topological notion of connectedness) if for every two points $a, b : X$ there \emph{merely} exists a path between them; this is expressed as the (propositionally) truncated type $\brck{a =_X b}$ being inhabited. An \emph{explicit} witness $p : a =_X b$ would imply having a canonical choice of path, a form of constructive existence which is stronger than mere existence.

Some types already have a truncated higher-order $\infty$-groupoid structure of their own, truncating them again has no effect. A type $Z$ for which $\trunc n Z \eqvsym Z$ is called an $n$-type. In this article we will mainly encounter $(-1)$-types and $0$-types. These are called \emph{propositions} and \emph{sets} respectively, the types of propositions and sets are denoted by $\prop$ and $\set$. Two terms of a proposition are always equal, e.g. for $p, q : \brck{a =_X b}$ it holds that $p = q$\,, and sets are homotopy equivalent to a collection of points.

\subsubsection{Loop spaces and homotopy groups}
We briefly recall the definitions of loop spaces, homotopy groups, and induced maps on these.

The term $\rrefl n x$ denotes the constant $n$-dimensional path (also called an $n$-\emph{cell}) from $x : X$ to itself. It is defined recursively by
\[
    \rrefl x {n+1} \defeq \refl{\rrefl x n} : (\rrefl x n = \rrefl x n)~~~~\text{with}~~~~\rrefl x 0 \defeq x~,
\]
so $\rrefl x 1 \jdeq \refl x : (x =_X x)$ denoted the constant path and $\rrefl x 2 : (\refl x =_X \refl x)$ denotes the constant homotopy between the constant path $\refl x$ and itself.

Given a pointed type $(X,x_0)$, the space $\Omega^n(X,x_0) \defeq (\rrefl {x_0} {n-1} = \rrefl {x_0} {n-1})$ of $n$-cells is called the $n$-th \emph{loop space} of $(X,x_0)$\,. It is a pointed type in itself, with the constant $n$-cell $\rrefl{x_0}n : \Omega^n(X,x_0)$ as the designated point. Thus, we have $\Omega^{n+1}(X,x_0) \defeq \Omega(\Omega^n(X,x_0))$\,.
The \emph{homotopy groups} $\pi_n(X,x_0)$ are defined as the \emph{sets} of $n$-dimensional loops, meaning that $\pi_n(X,x_0) \defeq \trunc 0 {\Omega^n(X,x_0)}$\,.

Given a pointed map $f : (X,x_0) \pto (Y,y_0)$ with $w_f : f(x_0) = y_0$ as its proof of pointedness (which alternatively can also be denoted as a pair $(f,w_f) : (X,x_0) \pto (Y,y_0)$), there are induced maps 
$f_* : \Omega(X,x_0) \to \Omega(Y,y_0)$ and $f_* : \pi_1(X,x_0) \to \pi_1(Y,y_0)$ given by
\[
    f_*(p) \defeq w_f^{-1} \ct \apfunc f (p) \ct w_f~~~~\text{and}~~~~f_*(\tproj 0 p) \defeq \mathopen{}|w_f^{-1} \ct \apfunc f (p) \ct w_f|_0\mathclose{}~~~~\text{with }p : (x_0 =_X x_0)~.
\]
Whichever version of $f_*$ is meant should be clear from context. Note that it suffices to define $f_* : \pi_1(X,x_0) \to \pi_1(Y,y_0)$ for terms of the form $\tproj 0 p$ since the codomain is a set, see the next section.

\subsection{Dealing with truncations}\label{sec:dealing-truncations}

The topological results in this article are about connected spaces, so we are constantly confronted with the fact that we only have truncated paths $\brck{a =_X b}$\,. Although they come from non-truncated paths, we cannot freely use them as such.

The basic way to deal with truncated types is by using their induction principle: if the goal is to prove an $n$-type, we may `strip' the truncation-bars from a truncated type $\trunc n Z$ and use it as if it were the non-truncated type. In practice, this means that an $n$-truncated term $z : \trunc n Z$ can be assumed to be of the form $z \jdeq \tproj n {z'}$ with $z' : Z$\,. So, if $W$ is an $n$-type, we can define maps $\trunc n Z \to W$ by only specifying the output on terms of the form $\tproj n z$\,. 

The absence of canonical paths is also an issue in the classical algebraic topology when doing constructions. There, the solution is to take an arbitrary path for the construction and to then show that final result does not dependent of the specific path chosen. We can do something similar in HoTT, with the caveat that the constructed object's type has to be a \emph{set}. The magic ingredient is a technique called \emph{extension by weak constancy}.

\begin{lemma}[Extension by weak constancy, cf.~generalization {\cite[Thm.\,1]{kraus2015weakextension}}] \label{lem:ext_weak_constancy}
Let $Z$ be a type and~$W$ a set. If $f : Z \to W$ satisfies $ f(z_1) =_W f(z_2)$ for all $z_1, z_2 : Z$\,, it can be extended to a map $g : \brck{Z} \to W$\, such that $g(\bproj z) \jdeq f(z)$ for all $z : Z$\,.
\end{lemma}

Extension by weak constancy is used as follows to construct objects from truncated paths. First, give the construction for an arbitrary, non-truncated path $p : a =_X b$\,, i.e. define a map $w : (a =_X b) \to W$\,, where $W$ denotes the type of object to construct. Provided that (i)~$W$ is a set, and (ii) $w(p) = w(q)$ for all $p, q : a =_X b$\,, the construction $w$ can be extended to a map $\overline{w} : \brck{a =_X b} \to W$\,. The inhabitant of the truncated path type, $\ast : \brck{a =_X b}$\,, is inserted to obtain the final object, $w^* \defeq \overline{w}(\ast) : W$\,.
Additionally, it holds that $w^* = w(p)$ for any explicit path $p : a =_X b$\,. This is because the equality is a proposition, so the truncation from $\ast$ can be stripped, from which it follows that
\[
    w^* \jdeq \overline{w}(\ast) \jdeq \overline{w}(\bproj{q}) \jdeq w(q) = w(p)~,
\]
where $\ast \jdeq \bproj{q}$ for some $q : a =_X b$\,.

\subsection{Notation}

We stick to the notation of the HoTT book, with some exceptions: we use $\ttype$ instead of $\type$ to denote the universe of types and we use $f_*$ to denote the induced maps on loop spaces and fundamental groups instead of $\Omega(f)$ and $\pi_1(f)$ in order to stay closer to the notation used by Hatcher \cite{hatcherAT}. As not to confuse $f_*$ with the shorthand notation for transport, like $p_* u = v$, we prefer to write out transport in full, namely as $(\transfib P p u = v)$\,.

\section{Classification of covering spaces} \label{sec:classification}

In this section we prove synthetic versions of the \emph{lifting criterion} and the \emph{classification of covering spaces}. With the right translations into HoTT, we were able to obtain low-level proofs that closely follow the their classical counterparts, e.g. those used by Hatcher \cite{hatcherAT}. Figuring out what the `right' translations were, however, took multiple attempts. Besides giving the final versions of definitions and statements, we also discuss some earlier versions and what trouble they caused us, so that others may learn from our experience.

\begin{note*} There is an official notion of `correctness' for translation of classical statements into HoTT. Homotopy type theory can be modeled in the topos of simplicial sets \cite{simplicial-model}, and simplicial sets have a geometric realization as CW complexes, i.e. actual topological spaces. Statements in HoTT can thus be interpreted as statements about CW complices, and so it can be checked that the translation of statements into HoTT is equivalent to the original statements under this interpretation.
Checking the correctness of translations in this way is not the purpose of this article. We rely on the homotopical intuition build up over time by the community and the \emph{a posteriori} justification provided by proofs of classical results.
\end{note*}

\subsection{Covering spaces in HoTT}

The study of covering spaces in HoTT was initiated by Hou and Harper \cite{favonia-covspaces}. They prove that every covering space of a pointed space $(X,x_0)$ corresponds to a set with a $\pi_1(X,x_0)$-action and they construct the universal covering space. We take their definition of a pointed covering space as a starting point:

\begin{definition}[cf.~{\cite[Def.\,1\&7]{favonia-covspaces}}]\label{def:cov-space}
    Let $(X,x_0)$ be a pointed type.
    A \textbf{\emph{covering space}} of $X$ is a set-valued type family $F : X \to \set$\,. If $F$ is equipped with a designated point $u_0 : F(x_0)$\,, the pair $(F,u_0)$ is called a \textbf{\emph{pointed}} covering space.
\end{definition}
Recall that type families naturally correspond to fibrations in HoTT. The requirement that~$F$ takes values in $\set$ guarantees the `sheetedness' of the covering space as it implies that each fiber is homotopic to a discrete collection of points.

Since types automatically inherit a notion of equality, we can also check a definition's correctness by inspecting its identity type.\footnote{This method of validating definitions in HoTT by checking if they produce the right type of equality was suggested by Egbert Rijke at the 2023 HoTT/UF workshop.} The classical classification theorem only distinguishes covering spaces up to fiberwise, base-point preserving homeomorphism, and, indeed, the lemma below shows that Hou and Harper's definition gives rise to the same notion of equality. (Note that the notion of homeomorphism can only be expressed as homotopy equivalence in HoTT.)

\begin{lemma}[Characterization of equality between pointed covering spaces] \label{lem:char_pcov}
    Let $(F_1,u_1)$ and $(F_2,u_2)$ be pointed covering spaces over a pointed type $(X,x_0)$\,. Then there is an equivalence
    \[
        (F_1,u_1) = (F_2,u_2) ~~~~~ \eqvsym ~~~ \sm{h\, :\, \prod_x F_1(x)\eqvsym F_2(x)} h (x_0, u_1) =_{F_2(x_0)} u_2~.
    \]
    This implies that two pointed covering spaces are equal if and only if there exists a basepoint-preserving, fiberwise equivalence between them.
\end{lemma}

\begin{proof}
A term $w : (F_1,u_1) = (F_2, u_2)$ is a path in the total space $\sum_X F$\,, so it is equivalent to a path between functions $h' : F_1 = F_2$ that satisfies
\[
    \transfib{F \mapsto F(x_0)}{h'}{u_1} = u_2~.
\]
By functional extensionality, $h'$ corresponds to the family $\happly(h') : \prod_x F_1(x) = F_2(x)$\,. Each $\happly(h',x) : F_1(x) = F_2(x)$ is a path between types (technically, between types with a proof that they are sets, but since being a set is a proposition in itself, this can be ignored), so by univalence $\happly(h',x)$ corresponds to an equivalence
\[
    \idtoeqv (\happly(h',x)) : F_1(x) \eqvsym F_2(x)~.
\]
Thus $h'$ is corresponds 1-to-1 to the fiberwise equivalence $h : \prod_{x:X} F_1(x) \eqvsym F_2(x)$ given by~$h(x,\blank) \defeq \idtoeqv (\happly(h',x))$\,.
By path induction, it holds that
\[
    h(x,\blank) \jdeq \idtoeqv (\happly(h',x)) = \transfib{F \to F(x)}{h'}{\blank}~,
\]
so the condition $(\transfib{F \mapsto F(x_0)}{h'}{u_1} = u_2)$ satisfied by $h'$ is equivalent to the statement $h(x_0,u_1) = u_2$\,, giving us an equivalence on the level of sigma types.
\end{proof}

\subsection{Lifting criterion}

The lifting criterion serves as a nice stepping stone towards the classification of covering spaces. Not only is it used in the classical proof of the classification theorem, it also allows us to practice with translating classical statements into homotopy type theory. First, we obtain a direct translation of the lifting criterion, but it turns out that this translation can be simplified whilst more closely reflecting the geometric ideas.

We define the lift of a map in homotopy type theory as follows.

\begin{definition} \label{def:ptd_lift}
    Let $f : (Y,y_0) \pto (X,x_0)$ be a pointed map with $w_f : f(y_0) =_X x_0$ its proof of pointedness. A \textbf{\emph{pointed lift}} of $f$ to the covering space $(F,u_0)$ over $(X,x_0)$ is a dependent map $\widetilde{f} : \prod_{y:Y} F(f(y))$ such that
    \[
        \transfib{F}{w_f}{\widetilde{f}(y_0)} =_{F(x_0)} u_0~.
    \]
\end{definition}

Instead of defining the lift as a dependent map $\widetilde{f} : \prod_{y:Y} F(f(y))$ --- meaning that $\widetilde{f}$ maps each point $y : Y$ to a point in the fiber over $f(y)$ --- we could have stayed closer to the classical definition and defined the lift as a map $\widetilde{f} : Y \to \sum_X F$ such that $w_{\widetilde f} : \fst \circ \widetilde{f} = f$\,. These formulations are equivalent and both types deserve to be called lifts, but we prefer to use the dependent map formulation because, in our experience, it was easier to work with.
We tried to use the classical formulation in an earlier attempt, but this required us to use transport along $w_{\widetilde f}$\, every time we had to compare terms in the fibers over $\fst(\widetilde f (y))$ and~$f(y)$\,. This does not happen when defining the lift as a dependent map, because then the equality $\fst \circ \widetilde{f} \jdeq f$ is a judgemental one.

Using Definition \ref{def:ptd_lift}, the lifting criterion can be formulated directly in homotopy type theory.

\begin{theorem}[direct translation, cf.~{\cite[Prop.\,1.33]{hatcherAT}}] \label{thm:lift-crit-direct}
    Let $(F,u_0)$ be a pointed covering space over a pointed type $(X,x_0)$\,. A pointed map $f : (Y, y_0) \pto (X,x_0)$\,, with $Y$ a connected type, can be lifted to pointed lift $\widetilde{f} : \prod_{y : Y} F(f(y))$ if and only if
    \begin{equation} 
        f_*(\pi_1(Y,y_0)) \subset {\fst}_*(\pi_1(\textstyle \sum_X F, (x_0,u_0)))~. \label{eq:lift_crit}
    \end{equation}
\end{theorem}

In proving Lemma \ref{thm:lift-crit-direct}, we found criterion (\ref{eq:lift_crit}) inconvenient to work with. The notation $f_*(\pi_1(Y,y_0))$ conceals multiple truncations --- a propositional-truncation to define the image of a map and a set-truncation for $\pi_1$ --- which hinder access to the paths themselves. Furthermore, Hou and Harper's definition of covering spaces does not involve the (pointed) total space $(\textstyle \sum_X F, (x_0,u_0))$\,, so we would prefer to have a version of criterion (\ref{eq:lift_crit}) that does not use it either.

Upon closer inspection, criterion (\ref{eq:lift_crit}) can be expressed more succinctly as a statement involving dependent paths. The criterion basically says that for any loop in type $x_0 =_X x_0$ of the form $f_*(p)$ with $p : y_0 =_Y y_0$\,, there exists a loop of type $(x_0, u_0) = (x_0, u_0)$ in the total space lying over $f_*(p)$\,. In HoTT, the latter is interpreted as the existence of a dependent loop
\[
    \transfib{F}{f_*(p)}{u_0} =_{F(x_0)} u_0~.
\]
This gives us an equivalent alternative to criterion (\ref{eq:lift_crit}) as stated in the lemma below. Note that the set-truncated $\pi_1$'s are also replaced by untruncated path types. This is possible because both statements are propositions.

\begin{lemma} \label{lem:eqv-lift-crit}
    Let $(F,u_0)$ be a pointed covering space over a pointed type $(X,x_0)$ and let $f:(Y,y_0) \pto (X,x_0)$ be a pointed map. The following propositions are equivalent:
    \begin{enumerate}[(i)]
        \item $f_*(\pi_1 (Y,y_{0})) \subset {\fst}_*(\pi_1(\textstyle \sum_X F, (x_0,u_0)))
        $\,;
        \item for all loops $p : y_0 =_Y y_0$ there exists a dependent loop of type\\ \phantom{~~}$(\transfib{F}{f_*(p)}{u_0} =_{F(x_0)} u_0)$\,.
    \end{enumerate}
\end{lemma}

The proof is a bit technical and involves a lot of truncations. We deal with these now so we do not have to deal with them when proving the lifting criterion.
\begin{proof}
    Since both conditions are propositions, it suffices to show that both statements imply each other.
    
    Assume that $f_*(\pi_1(Y,y_0)) \subset {\fst}_*(\pi_1(\textstyle \sum_X F, (x_0,u_0)))$ and let $p : y_0 =_Y y_0$\,. Naturally, the truncated loop $f_*(\tproj 0 p)$ lies in $f_*(\pi_1(Y,y_0))$\,, so by assumption it also lies in ${\fst}_*(\pi_1(\textstyle \sum_X F, (x_0,u_0)))$\,, meaning that there \emph{merely} exists a $q' : \pi_1(\sum_X F, (x_0, u_0))$ such that ${\fst}_*(q') = f_*(\tproj{0}{p})$\,. The goal is to show that $(\transfib{F}{f_*(p)}{u_0} =_{F(x_0)} u_0)$\,, and since $F(x_0)$ is a set, this is a proposition. Hence, we can strip the `merely' in `merely exists' and we can strip the truncation from $\pi_1(\sum_X F, (x_0, u_0)) \jdeq \trunc 0 {(x_0,u_0) = (x_0, u_0)}$ to obtain an explicit loop $q : (x_0,u_0) = (x_0, u_0)$ which satisfies ${\fst}_*(\tproj{0}{q}) = f_*(\tproj{0}{p})$\,.
    As $q$ is a loop in the total space, it can be split into its projection onto the base space $\apfunc{\fst}(q) : x_0 =_X x_0$ and a dependent loop of type $(\transfib{F}{\apfunc{\fst}(q)}{u_0} =_{F(x_0)} u_0)$ in the fiber $F(x_0)$\,. Therefore, to show that $(\transfib{F}{f_*(p)}{u_0} =_{F(x_0)} u_0)$\,, it suffices to show that 
    \[
        \apfunc{\fst}(q) = f_*(p)~.
    \]
    This follows from the fact that $q$ satisfies ${\fst}_*(\tproj{0}{q}) = f_*(\tproj{0}{p})$\,: using the definitions of the induced maps, this equality can be rewritten as $\tproj{0}{{\fst}_*(q)} = \tproj{0}{f_*(p)}$\,, which by Theorem~7.3.12 in \cite{hottbook}, is equivalent to $\brck{{\fst}_*(q) = f_*(p)}$\,. The truncation can be stripped, and the result implies that
    \[
        \apfunc{\fst}(q) = \refl{x_0}^{-1} \ct \apfunc{\fst}(q) \ct \refl{x_0} \jdeq {\fst}_*(q) = f_*(p)~.
    \]

    Conversely, assume that for all loops $p : y_0 =_Y y_0$ there exists a dependent path of the form $(\transfib{F}{f_*(p)}{u_0} =_{F(x_0)} u_0)$\,. Let $r : \pi_1(X,x_0)$ and assume that $r$ lies in the image $f_*(\pi_1(X,x_0))$\,, meaning that there merely exists a $p' : \pi_1(Y,y_0)$ such that $f_*(p') = r$\,. The goal is to show the proposition that $r'$ lies in ${\fst}_*(\pi_1(\textstyle \sum_X F, (x_0,u_0)))$\,, so we can again strip the truncations in `the mere existence of $p'$' to obtain an explicit loop $p : y_0 =_Y y_0$ that satisfies $f_*(\tproj 0 p) = r$\,. To show that $r$ lies in ${\fst}_*(\pi_1(\textstyle \sum_X F, (x_0,u_0)))$\,, it suffices to construct a loop $q : (x_0, u_0) = (x_0,u_0)$ that satisfies ${\fst}_* (\tproj 0 q) = r$\,. Since $r = f_*(\tproj 0 p)$\,, this latter property can be replaced by ${\fst}_* (\tproj 0 q) = f_*(\tproj 0 p)$\,. From here we can ignore the loop~$r$\,, making this part of the proof more similar to the previous part.
    To construct a loop $q : (x_0, u_0) = (x_0,u_0)$ it suffices to provide a loop $q_1 : x_0 =_X x_0$ in the base space and a dependent loop of type $(\transfib{F}{q_1}{u_0} =_{F(x_0)} u_0)$ over $q_1$\,. Choose $q_1 \defeq f_*(p)$\,, then the dependent loop of type $(\transfib{F}{f_*(p)}{u_0} =_{F(x_0)} u_0)$ exists by assumption. Furthermore, by construction it holds that $\apfunc \fst (q) \jdeq q_1 \jdeq f_*(p)$\,, so it holds that
    \[
        {\fst}_* (\tproj 0 q) \jdeq \mathopen{}|\refl{x_0}^{-1} \ct \apfunc{\fst}(q) \ct \refl{x_0} |_0 \mathclose{} = \mathopen{}| \apfunc \fst (q) |_0\mathclose{} = \tproj 0 {f_*(p)} \jdeq f_*(\tproj 0 p)~.
    \]  
\end{proof}

Replacing criterion (\ref{eq:lift_crit}) in the direct translation of the lifting criterion (Theorem \ref{thm:lift-crit-direct}) with proposition (ii) from Lemma \ref{lem:eqv-lift-crit}, we obtain a version of the lifting criterion that is better suited to the language of homotopy type theory. We prove this version and, hence, also the direct translation. The proof closely follows the argument used by Hatcher~\cite{hatcherAT}.

\begin{theorem}[Lifting criterion, cf.~{\cite[Prop.\,1.33]{hatcherAT}}] \label{thm:lift-crit-2}
    Let $(F,u_0)$ be a pointed covering space over a pointed type $(X,x_0)$\,. A pointed map $f : (Y, y_0) \pto (X,x_0)$\,, with $Y$ a connected type, can be lifted to a pointed lift $\widetilde{f} : \prod_{y : Y} F(f(y_0))$ if and only if for all loops $p : y_0 =_Y y_0$ there exists dependent loop
    \[
        \transfib{F}{f_*(p)}{u_0} =_{F(x_0)} u_0 ~,
    \]
    meaning that there exists a loop from $u_0$ to $u_0$ over $f_*(p)$ in $F$\,. 
\end{theorem}

\begin{proof}
Before we start, note that the point $x_0 : X$ is arbitrary, and hence we can perform path induction on the path $w_f : f(y_0) =_X x_0$ that encodes $f$'s pointedness. So assume that we have a judgmental equality $x_0 \jdeq f(y_0)$\,. This lets us replace every occurrence of $x_0$ in the theorem's statement by $f(y_0)$\,. Moreover, we also get to assume that the path $w_f$ itself is judgmentally equal to the constant path, i.e. $w_f \jdeq \refl{f(y_0)}$\,.

Assume that there exists a dependent loop of type $(\transfib{F}{f_*(p)}{u_0}=_{F(x_0)}u_0)$ for every loop $p : y_0 =_Y y_0$\,.
Let $y :Y$\,. Since $Y$ is connected, the mere path type $\brck{y_0 =_Y y}$ is inhabited, which we use to define the lift $\widetilde{f}(y)$ like Hatcher:
first, we take an arbitrary path $q : y_0 =_Y y$ to construct $\widetilde{f}(y)$ by transporting the point $u_0 : F(x_0)$ to the fiber $F(f(y))$ along the path $\apfunc{f}(q) : f(y_0) = f(y)$\,; we then show that $\widetilde{f}(y)$ is well-defined, i.e. that the resulting construction did not depend on the specific choice of path $q : y_0 =_Y y$\,. This way of constructing $\widetilde{f}(y)$ is justified because of extension by weak constancy (Lemma \ref{lem:ext_weak_constancy}) as explained in Section~\ref{sec:dealing-truncations}.

Let $q : y_0 =_Y y$ be an arbitrary path; this gives a path $\apfunc f (q) : f(y_0) =_X f(y)$. We then define the lift $\widetilde f (y)$ as
\[
    \widetilde f (y) \defeq \transfib{F}{\apfunc{f}(q)}{u_0} : F(f(y))~.
\]
It remains to show that this construction does not depend on the choice of path of type $y_0 =_Y y$\,. Let $q_1, q_2 : y_0 =_Y y$\,, we need to show that
\[
    \transfib{F}{\apfunc{f}(q_1)}{u_0} = \transfib{F}{\apfunc{f}(q_2)}{u_0}~.
\]
We again follow Hatcher's proof. Remark that $\apfunc{f}(q_1)\ct\apfunc{f}(q_2)^{-1}$ is a loop in the type $f(y_0) =_X f(y_0)$ of the form $f_*(p)$ for some $p : y_0 =_Y y_0$\,, namely
\[
    \apfunc{f}(q_1) \ct \apfunc{f}(q_2)^{-1} = \apfunc{f}(q_1 \ct q_2^{-1}) = \refl{f(y_0)}^{-1} \ct \apfunc{f}(q_1 \ct q_2^{-1}) \ct \refl{f(y_0)} \jdeq f_*(q_1 \ct q_2^{-1})~.
\]
Hence, by assumption we have a dependent loop over $\apfunc{f}(q_1) \ct \apfunc{f}(q_2)^{-1}$\,,
\[
    \transfib{F}{\apfunc{f}(q_1) \ct \apfunc{f}(q_2)^{-1}}{u_0} =_{F(x_0)} u_0~.
\]
Transporting both sides back along $\apfunc{f}(q_2)$\,, we see that $\apfunc{f}(q_1)$ and $\apfunc{f}(q_2)$ both send $u_0$ to the same point in $F(f(y))$\,:
\[
    \transfib{F}{\apfunc{f}(q_1)}{u_0} = \transfib{F}{\apfunc{f}(q_2)}{u_0}~.
\]
This concludes the construction of the lift $\widetilde f(y)$\,. 

Since $w_f \jdeq \refl{f(y_0)}$\,, proving that $\widetilde{f}$ is a pointed lift reduces to showing that $\widetilde{f}(y_0) = u_0$\,. To compute $\widetilde{f}(y_0)$\,, we utilize that we have access to an explicit loop $\refl{y_0} : y_0 =_Y y_0$\,. This gives us that
    \begin{align*}
    \widetilde{f}(y_0) = \transfib{F}{\apfunc{f}(\refl{y_0})}{u_0} \jdeq \transfib{F}{\refl{f(y_0)}}{u_0} \jdeq u_0~.
    \end{align*}

Conversely, assume that a pointed lift $\widetilde{f}$ exists. Let $p : y_0 =_Y y_0$\,, then we need to show that there exists a dependent loop of type $(\transfib{F}{f_*(p)}{u_0} =_{F(x_0)} u_0)$\,. The path $p$ gives rise to a dependent equality between $\widetilde f (y_0)$ and itself over $p$ in the family $F \circ f$\,, namely
\[
    \apdfunc{\widetilde f}\,(p) : \transfib{F \circ f}{p}{\widetilde{f}(y_0)} =_{F(f(y_0))} \widetilde{f}(y_0)~.
\]
By Lemma 2.3.10 in the HoTT book \cite{hottbook}, transport along $p$ in $F \circ f$ equals transport along~$\apfunc f (p)$ in $F$\,, so we get a dependent loop at $\widetilde f(y_0)$ in $F$\,,
\[
    \transfib{F}{\apfunc{f}{(p)}}{\widetilde{f}(y_0)} =_{F(f(y))} \transfib{F \circ f}{p}{\widetilde{f}(y_0)} =_{F(f(y))} \widetilde{f}(y_0)~.
\]
Since $\widetilde{f}$ is pointed, i.e. $\widetilde{f}(y_0) = u_0$\,, this gives the dependent loop over $f_*(p)$ that was sought:
\begin{align*}
    &\transfib{F}{f_*(p)}{u_0} \jdeq \transfib{F}{\refl{f(y_0)}^{-1} \ct \apfunc f (p) \ct \refl{f(y_0)}}{u_0} \\
    &~~~~~~= \transfib F {\apfunc f (p)}{u_0} = \transfib{F}{\apfunc f (p)}{\widetilde{f}(y_0)} = \widetilde f (y_0) = u_0~.
\end{align*}
\end{proof}

\subsection{Classification theorem}

After translating the lifting criterion, we are ready to state and prove the classification of covering spaces in homotopy type theory. The theorem is given below. Note how the subgroup associated to a covering space is not defined as the image of the covering map, but is given in terms of dependent loops, like with the reformulated lifting criterion in Theorem~\ref{thm:lift-crit-2}.

\begin{theorem}[Classification, cf.~{\cite[first half of Thm.\,1.38]{hatcherAT}}] \label{thm:classification}
    Let $(X,x_0)$ be a pointed type, then there exists an equivalence between pointed, connected covering spaces $(F,u_0)$ over $(X,x_0)$ and subgroups of $\pi_1(X,x_0)$\,, obtained by associating to $(F,u_0)$ the subgroup $H_{(F,u_0)}$ given by
    \[
        H_{(F,u_0)} (\tproj 0 p) \defeq (\transfib{F}{p}{u_0} =_{F(x_0)} u_0 )~,
    \]
    meaning that $\tproj{0}{p} : \pi_1(X,x_0)$ belongs to $H_{(F,u_0)} : \pi_1(X,x_0) \to \prop$ if there exists a loop from~$u_0$ to $u_0$ lying over $p$ in $F$\,.
\end{theorem}

The theorem classifies \emph{connected} covering spaces. Again, we follow Hou and Harper \cite{favonia-covspaces}'s definition:

\begin{definition}
    A covering space $F : X \to \set$ is called \textbf{\emph{connected}} if its total space $\sum_X F$ is connected.
\end{definition}

Both the HoTT proof and the classical proof consists of two parts: showing that the association in Theorem \ref{thm:classification} is injective and surjective. Surjectivity is shown using the universal covering space $P$ constructed by Hou and Harper \cite{favonia-covspaces}. To show injectivity, the classical proof uses the lifting criterion to construct maps between covering spaces $\widetilde{X_1}$ and $\widetilde{X_2}$ by lifting the respective projections:
\[
    \begin{tikzcd}
        \widetilde{X_1} \arrow[r, dashed, "\widetilde{p_1}"] \arrow[rd, "p_1"'] 
            & \widetilde{X_2} \arrow[r, dashed, "\widetilde{p_2}"] \arrow[d, "p_2"] & \widetilde{X_1} \arrow[ld, "p_1"] \\
            & X
    \end{tikzcd}
\]
Unfortunately, we cannot use the lifting criterion for the same purpose in the HoTT-setting, as the codomains of the maps we wish to construct are no longer spaces like $\widetilde{X_i}$\,, but families $F_i : X \to \set$\,. Therefore, we prove two new lemmas that give conditions for the existence and uniqueness of maps between covering spaces defined as families of sets. They correspond to the lifting criterion and the unique lifting property in the classical theory.

\begin{lemma} \label{lem:fiberwise_map_cov}
    Let $(F_1,u_1)$ and $(F_2,u_2)$ be pointed covering spaces over a pointed type $(X,x_0)$ with $F_1$ connected. Then there exists a fiberwise map $h : \prod_{x:X} F_1(x) \to F_2(x)$ that preserves the basepoint, meaning that $h(x_0, u_1) = u_2$\,, if and only if for all loops $p : x_0 =_X x_0$ we have
    \[
         ( \transfib{F_1}{p} {u_1} =_{F_1(x_0)} u_1 ) ~~ \longrightarrow ~~ ( \transfib{F_2}{p}{u_2} =_{F_2(x_0)} u_2 )~,
    \]
    meaning that existence of a loop from $u_1$ to $u_1$ over $p$ in $F_1$ implies existence of a loop from~$u_2$ to $u_2$ over $p$ in $F_2$\,.
\end{lemma}

\begin{lemma}\label{lem:fiberwise_map_unique}
    Let $(F_1,u_1)$ and $(F_2,u_2)$ be pointed covering spaces over a pointed type $(X,x_0)$ with $F_1$ connected. Then any two fiberwise maps $h_1, h_2 : \prod_{x : X} F_1(x) \to F_2(x)$ are equal if they coincide in a single point, e.g. if $h_1(x_0, u_1) = h_2(x_0, u_1)$\,.
\end{lemma}

Since Lemma~\ref{lem:fiberwise_map_cov} serves as a `family of sets'-based version of the lifting criterion, its proof is also similar to that of Theorem~\ref{thm:lift-crit-2}.

\begin{proof}[Proof (Lemma \ref{lem:fiberwise_map_cov})]
    Assume that for all for all loops $p : x_0 =_X x_0$\,, existence of a dependent loop of type $( \transfib{F_1}{p} {u_1} =_{F_1(x_0)} u_1 )$ in $F_1$ implies existence of a dependent loop of type $( \transfib{F_2}{p} {u_2} =_{F_2(x_0)} u_2 )$ in $F_2$\,.
    Let $x : X$ and $u : F_1(x)$\,. We construct $h(x,u)$ in the same way as the lift in the proof of the lifting criterion (Theorem \ref{thm:lift-crit-2}), namely as the transport of the designated point $u_2 : F_2(x_0)$ along a path to the fiber $F_2(x)$\,. The path which $u_2$ is transported along is given in terms of a path in the total space of $F_1$\,. Since~$F_1$ is connected, we only have the mere existence of such paths, so we need to show that the construction of $h(x,u)$ is path-independent. Like in the lifting criterion, we again use extension by weak constancy (Lemma \ref{lem:ext_weak_constancy}) to replicate this construction method in HoTT.

    Take an arbitrary path of type $(x_0,u_1) = (x,u)$ in the total space of $F_1$\,. This is equivalent to an arbitrary path $q : x_0 =_X x$ such that $(\transfib{F_1}{q}{u_1} =_{F_1(x)} u)$\,. Since transporting~$u_1$ along $q$ yields $u$\,, it seems reasonable that transport of $u_2$ along $q$ should give $h(x,u)$\,. So, define $h(x,u)$ as:
    \[
        h(x,u) \defeq \transfib{F_2}{q}{u_2} : F_2(x)~.
    \]
    
    It remains to show that any two paths in $(x_0,u_1) = (x,u)$ produce the same point~$h(x,u)$\,. It suffices to show that any two paths $q_1, q_2 : x_0 =_X x$ satisfying $(\transfib{F_1}{q_i}{u_1} =_{F_1(x)} u)$ yield the same value for $h(x,u)$\,. Let $q_1, q_2$ be such paths. Like in the proof of the lifting criterion (Theorem \ref{thm:lift-crit-2}), $(q_1 \ct q_2^{-1})$ is a loop of type $x_0 =_X x_0$ and it holds that 
    \[
        \transfib{F_1}{q_1 \ct q_2^{-1}}{u_1} =_{F_1(x_0)} u_1 ~\text{,~~~~and so~~~~}\transfib{F_2}{q_1 \ct q_2^{-1}}{u_2} =_{F_2(x_0)} u_2~.
    \]
    It follows that $(\transfib{F_2}{q_1}{u_2} =_{F_2(x_0)} \transfib{F_2}{q_2}{u_2})$\,, so $h(x,u)$ is well-defined.

    To prove pointedness, we utilize that we have access to an explicit loop $\refl{x_0} : x_0 =_X x_0$ which satisfies $(\transfib{F_1}{\refl{x_0}}{u_1} \jdeq u_1)$\,. This allows us to compute $h(x_0,u_1)$\,:
    \[
        h(x_0,u_1) = \transfib{F_2}{\refl{x_0}}{u_2} \jdeq u_2~.
    \]

    Conversely, let $h : \prod_{x : X} F_1(x) \to F_2(x)$ be a basepoint-preserving, fiberwise map between covering spaces, let $p : x_0 =_X x_0$ be a loop and assume there exists a dependent loop of type $(\transfib{F_1}{p}{u_1} =_{F_1(x_0)} u_1)$\,. By Lemma \ref{lem:comm_with_transport}, the family $h(x,\blank)$ commutes with transport, which implies the existence of a dependent loop at $h(x_0,u_1)$ over $p$ in $F_2$\,, namely
    \[
        \transfib{F_2}{p}{h(x_0, u_1)} = h(x_0, \transfib{F_1}{p}{u_1}) = h(x_0, u_1)~.
    \]
    Since $h(x_0, u_1) = u_2$\,, we thus have a dependent loop of type $(\transfib{F_2}{p}{u_2} =_{F_2(x_0)} u_2)$\,, which is what we needed.
\end{proof}

\begin{proof}[Proof (Lemma \ref{lem:fiberwise_map_unique})]
    Let $h_1, h_2 : \prod_{x : X} F_1(x) \to F_2(x)$ be fiberwise maps and w.l.o.g. assume they coincide in $u_1$, so $h_1(x_0,u_1) = h_2(x_0,u_2)$\,. Let $x : X$ and $u : F_1(x)$ be arbitrary, the goal is to prove the proposition $h_1(x,u) =_{F_2(x)} h_2(x,u)$\,. $F_1$ is connected, so the mere path type $\brck{(x_0,u_1) = (x,u)}$ is inhabited. As the goal is a proposition, the truncation can be stripped, leaving an explicit path. This path is equivalent to a path $p : x_0 =_X x$ that satisfies $(\transfib{F_1}{p}{u_1} =_{F_1(x_0)} u)$\,. Lemma $\ref{lem:comm_with_transport}$ implies that the families $h_i(x,\blank)$ commute with transport, so
    \[
        h_i(x,u) = h_i(x,\transfib{F_1}{p}{u_1}) = \transfib{F_2}{p}{h_i(x_0,u_1)}
    \]
    for $i = 1,2$\,. By the assumption $h_1(x_0,u_1) = h_2(x_0,u_1)$\,, we thus have $h_1(x,u) = h_2(x,u)$ for arbitrary $x$ and $u$\,.
\end{proof}

Using Lemma \ref{lem:fiberwise_map_cov} and \ref{lem:fiberwise_map_unique} as replacements for the lifting criterion and the unique lifting property, we give a proof for the classification theorem in HoTT. Again, the proof itself closely follows the argument used by Hatcher \cite{hatcherAT}. 

\begin{proof}[Proof (Theorem \ref{thm:classification})]
    We need to show that the association $(F,u_0) \mapsto H_{(F,u_0)}$ is an equivalence, meaning that it is both injective and surjective.

    The proof of surjectivity uses the universal covering space $P : X \to \set$ constructed by Hou and Harper \cite{favonia-covspaces}\,; the fibers $P(x)$ are defined as the sets of paths from $x_0$ to $x$
    \[
        P(x) \defeq \trunc{0}{x_0 =_X x}~.
    \]
    Let $H : \pi_1(X,x_0) \to \prop$ be a subgroup. We follow Hatcher \cite{hatcherAT} for the construction of a covering space $F_H$\,. Let $\sim_x$ denote the relation on $\trunc{0}{x_0 =_X x}$ defined by
    \[
        q_1 \sim_x q_2 \defeq H(q_1 \ct q_2^{-1})~~\text{with}~q_1, q_2 : \trunc{0}{x_0 =_X x}~.
    \]
    Since $H$ is closed under group operations, $\sim_x$ is an equivalence relation. The covering space $F_H : X \to \set$ is then defined as the set-quotient of $P$ w.r.t $\sim$\,:
    \[
        F_H(x) \defeq P(x)/\sim_x ~~ \text{with designated point} ~u_H \defeq [\,\tproj{0}{\refl{x_0}}] : F_H(x_0)~.
    \]
    Connectedness of $F_H$ follows along similar lines as Lemma 3.11.8 in the HoTT book \cite{hottbook}. In order to reason about the terms $v : F_H(x)$ as being given by representatives $v \jdeq [p']$ with $p' : \brck{x_0 =_X x}$\,, we do need to perform induction on the quotient type. The (higher) coherence conditions required are satisfied because a type being connected is a proposition.
    
    We need to show that the subgroup $H_{(F_H, u_H)}$ is equal to $H$\,. This is equivalent to showing that for all loops $p : x_0 =_X x_0$ there is an equivalence
    \[
         H(\tproj{0}{p})~~~~ \eqvsym ~~~~ \bigl( \transfib{F_H}{p}{u_H} =_{F_H(x_0)} u_H \bigr) ~~~~(\,\jdeq~ H_{(F_H, u_H)}(\tproj 0 p)\,)~.
    \]
    (Being an equivalence is a proposition, so the truncation from loops in $\pi_1(X,x_0) \jdeq \trunc 0 {x_0 =_X x_0}$ can be stripped.) Let $p : x_0 =_X x_0$\,. By definition of $(F_H, u_H)$\,, the dependent loop type $(\transfib{F_H}{p}{u_H} =_{F_H(x_0)} u_H)$ can be rewritten as $[\,\tproj{0}{p}] =_{F_H(x_0)} [\,\tproj{0}{\refl{x_0}}]$ because $u_H \jdeq [\,\tproj{0}{\refl{x_0}}]$ and 
    \begin{align*}
        \transfib{F_H}{p}{u_H} 
            &\jdeq \transfib{F_H}{p}{[\,\tproj{0}{\refl{x_0}}]} = [\,\tproj{0}{\transfib{x \mapsto x_0 =_X x}{p}{\refl{x_0}}}] = [\,\tproj{0}{p}]~,
    \end{align*}
    where we use that the family of maps $[\,\tproj 0 \blank] : (x_0 =_X x) \to F_H(x)$ commutes with transport (Lemma \ref{lem:comm_with_transport}).
    Since $\sim_{x_0}$ is an equivalence relation, $[\,\tproj{0}{p}] =_{F_H(x_0)} [\,\tproj{0}{\refl{x_0}}]$ is equivalent to $\tproj{0}{p} \sim_{x_0} \tproj{0}{\refl{x_0}}$\,, which, by definition, is equivalent to $H(\tproj{0}{p} \ct \tproj{0}{\refl{x_0}}^{-1})$ and hence $H(\tproj 0 p)$\,. Thus, the subgroups $H_{(F_H, u_H)}$ and $H$ are equal.

    We now show that the association $(F,u_0) \mapsto H_{(F,u_0)}$ is injective. Assume that for two connected, pointed covering spaces $(F_1,u_1)$ and $(F_2,u_2)$\,, it holds that $H_{(F_1,u_1)} = H_{(F_2,u_2)}$\,. By definition of these subgroups, this implies that for all loops $p : x_0 =_X x_0$ we have
    \[
        \bigl( \transfib{F_1}{p} {u_1} =_{F_1(x_0)} u_1 \bigr) ~~ \eqvsym ~~ \bigl( \transfib{F_2}{p}{u_2} =_{F_2(x_0)} u_2 \bigr)~.
    \]
    (Again, truncations can be stripped from loops in $\pi_1(X,x_0)$ as the goal $(F_1,u_1) = (F_2,u_2)$ is a proposition.) By Lemma \ref{lem:fiberwise_map_cov} there exist fiberwise maps $h_{12} : \prod_{x : X} F_1(x) \to F_2(x)$ and $h_{21} : \prod_{x:X} F_2(x) \to F_1(x)$ such that $h_{12}(x_0,u_1) = u_2$ and $h_{21}(x_0,u_2) = u_1$\,. By Lemma~\ref{lem:fiberwise_map_unique}, it holds that $h_{21} \circ h_{12} = \idfunc[F_1]$ and $h_{12} \circ h_{21} = \idfunc[F_2]$\,, because these maps coincide in the designated points $(x_0,u_1)$ and $(x_0, u_2)$\,, namely 
    \[
        h_{21}(x_0, h_{12}(x_0, u_1)) = u_1 \jdeq \idfunc[F_1](x_0,u_1) ~~~~ \text{and} ~~~~ h_{12}(x_0, h_{21}(x_0, u_2)) = u_2 \jdeq \idfunc[F_2](x_0,u_2)~.
    \]
    Hence, $h_{12}$ and $h_{21}$ are each other's inverses, and so $h_{12}$ is a basepoint-preserving, fiberwise equivalence from $(F_1,u_1)$ to $(F_2,u_2)$. By Lemma \ref{lem:char_pcov} this means the pointed covering spaces are equal, thus proving the classification theorem. 
\end{proof}

\section{Canonical change of basepoint} \label{sec:change_of_basepoint}

To develop the necessary theory to prove Exercise 3.3.11 from Hatcher's \emph{Algebraic Topology}~\cite{hatcherAT} --- our initial goal --- we needed the existence of change-of-basepoint isomorphisms between homotopy groups of spheres, $\pi_n(S^n,x) \cong \pi_n(S^n,y)$\,. Such isomorphisms were needed to define the degree of a non-pointed map $S^n \to S^n$\,. The degree is easy to define for pointed maps, but for non-pointed maps you need a consistent way to associate $\pi_n(S^n, f(\base))$ with $\pi_n(S^n, \base)$\,. In this section we prove some classical results on the existence of change-of-basepoint isomorphisms for connected spaces in general.

In classical homotopy theory, any path $p$ from $a$ to $b$ in a topological space $X$ induces a change-of-basepoint isomorphism between homotopy groups $\pi_n(X,a) \cong \pi_n(X,b)$\,. The isomorphism depends on the homotopy class of the path $p$\,.
In the case that $X$ is simply-connected, the isomorphism can be considered canonical --- there is only one class of paths from $a$ to $b$\,.

In homotopy type theory, this change-of-basepoint isomorphism is given by transport along a path $p : a =_X b$, but usually we do not have access to such a path explicitly. If~$X$ is connected, all we have is that the mere path type $\brck{a =_X b}$ is inhabited. However, if transport along any specific path of type $a =_X b$ yields the same isomorphism, we can still obtain an explicit isomorphism from $\brck{a =_X b}$ via extension by weak constancy (Lemma \ref{lem:ext_weak_constancy}). The result we call a \textbf{canonical} change-of-basepoint isomorphism, since it is independent on the path $p : a =_X b$\,.

\begin{note*}
Because extension by weak constancy requires the constructed object's type to be a set, we cannot do away with the set-truncations in $\pi_n(X,x) \jdeq \trunc 0 {\Omega^n (X,x)}$ as we often could in Section \ref{sec:classification}.
\end{note*}

The condition that transport along every path in $a =_X b$ results in the same map $\pi_n(X,a) \to \pi_n(X,b)$ is equivalent to the condition that the fundamental group $\pi_1(X,a)$ acts trivially on the higher homotopy groups $\pi_n(X,a)$\,. The latter is a well-studied property.

\begin{definition}[$\pi_1$-action] \label{def:pi1-action}
    Let $(X,x_0)$ be a pointed type. The action of $\pi_1(X,x_0)$ on the higher homotopy groups $\pi_n(X,x_0)$ is defined on truncated loops $\tproj{0}{p} : \pi_1(X,x_0)$ by
    \[
        \tproj{0}{p} \,.\, u  \defeq \transfib{\pi_n(X,\blank)}{p^{-1}}{u}~,
    \]
    with $u : \pi_n(X,x_0)$\,. The inversion $p^{-1}$ is to obtain a \emph{left} action. The action is called \emph{trivial} if multiplication by any term $p' : \pi_1(X,x_0)$ leaves $u : \pi_n(X,x_0)$ unchanged, i.e. $p' \,.\, u = u$\,.
\end{definition}

\begin{lemma} \label{lem:eqv_trivial_pi1}
    Let $X$ be a type with points $a, b : X$. Then the following propositions hold:
    \begin{enumerate}[(i)]
        \item for all paths $q_1, q_2 : a =_X b$ we have $(\transfib{\pi_n(X,\blank)}{q_1}{\blank} = \transfib{\pi_n(X,\blank)}{q_2}{\blank})$\,
            if the $\pi_1(X,a)$-action on $\pi_n(X,a)$ is trivial.
            
        \item the $\pi_1(X,a)$-action on $\pi_n(X,a)$ is trivial if the mere path type $\brck{a =_X b}$ is inhabited and for all $q_1, q_2 : a =_X b$ it holds that $(\transfib{\pi_n(X,\blank)}{q_1}{\blank} = \transfib{\pi_n(X,\blank)}{q_2}{\blank})$\,.
    \end{enumerate}
\end{lemma}

\begin{corollary}
    Let $X$ be a type with points $a, b : X$\,. There exists a change-of-basepoint isomorphism $\varphi : \pi_n(X,a) \cong \pi_n(X,b)$ which is \emph{canonical}, in the sense that for all $p : a =_X b$ it holds that $\varphi = \transfib{\pi_n(X,\blank)}{p}{\blank}$\,, if and only if the $\pi_1(X,a)$-action on $\pi_n(X,a)$ is trivial.
\end{corollary}

\begin{proof}[Proof (Lemma \ref{lem:eqv_trivial_pi1})]
    We first show part (i). Assume that the $\pi_1(X,a)$ action on $\pi_n(X,a)$ is trivial.
    Since the goal is to show a proposition, this assumption can be stated as
    \[
        \transfib {\pi_n(X,\blank)}{p^{-1}}{u} = u~,
    \]
    for any loop $p : a =_X a$ and $u : \pi_n(X,a)$\,. Let $q_1, q_2 : a =_X b$\,, then $(q_1 \ct q_2^{-1})$ is a loop of type $a =_X a$\,, so by assumption it holds that
    \[
        \transfib{\pi_n(X,\blank)}{q_2^{-1}}{\transfib{\pi_n(X,\blank)}{q_1}{u}} = \transfib{\pi_n(X,\blank)}{q_1 \ct q_2^{-1}}{u} = u~,
    \]
    for all $u : \pi_n(X,a)$\,. The result follows by applying $\transfib{\pi_n(X,\blank)}{q_2}{\blank}$ to both sides.

    Now for part (ii). Assume the mere path type $\brck{a =_X b}$ is inhabited and that for all paths $q_1, q_2 : a =_X b$ it holds that $(\transfib{\pi_n(X,\blank)}{q_1}{\blank} = \transfib{\pi_n(X,\blank)}{q_2}{\blank})$\,.
    Now let~$p' : \pi_1(X,a)$ and $u : \pi_n(X,a)$\,. Since the goal is to show a proposition, we may assume that $p' \jdeq \tproj 0 p$ for some loop $p : a =_X a$\,. We also strip the truncation from $\brck{a =_X b}$ to obtain an explicit path $p_{ab} : a =_X b$\,. By assumption, transport along both paths $(p^{-1} \ct p_{ab}) : a =_X b$ and $p_{ab} : a =_X b$ yields the same function, so
    \begin{align*}
        \transfib{\pi_n(X,\blank)}{&p_{ab}}{\transfib{\pi_n(X,\blank)}{p^{-1}}{u}} = \transfib{\pi_n(X,\blank)}{p^{-1} \ct p_{ab}}{u} \\
        &= \transfib{\pi_n(X,\blank)}{p_{ab}}{u}~,
    \end{align*}
    for all $u : \pi_n(X,a)$\,. Applying $\transfib{\pi_n(X,\blank)}{p_{ab}^{-1}}{\blank}$ to both sides yield the desired result that $(\transfib{\pi_n(X,\blank)}{p^{-1}}{u} = u)$\,.
\end{proof}

The following theorem collects some results on triviality of the $\pi_1$-action in homotopy type theory. Results (i) and (ii) cover the cases where $(X,x_0) \defeq (S^n,\base)$\,, giving us the canonical change-of-basepoint isomorphisms $\pi_n(S^n,x) \cong \pi_n(S^n,y)$ we needed to define the degree for non-pointed maps. In working on these results, we also managed to prove third, more complicated result, for which we were unable to find a reference in the classical theory.

\begin{theorem}\label{lem:trivial-action}
    Let $(X,x_0)$ be a pointed type.
\begin{enumerate}[(i)]
    \item If $X$ is simply-connected, then the action of $\pi_1(X,x_0)$ on $\pi_n(X,x_0)$ is trivial for all~$n \geq 1$; 
    \item The fundamental group $\pi_1(X,x_0)$ is abelian if and only if the action of $\pi_1(X,x_0)$ on itself is trivial;
    \item If merely for all loops $p, q:\Omega(X,x_0)\,$ it holds that $p \ct q = q \ct p$\,, then the action of~$\pi_1(X,x_0)$ on $\pi_n(X,x_0)$ is trivial for all $n \geq 1$\,.
\end{enumerate}
\end{theorem}

Results (i) and (ii) follow quickly in both the synthetic and classical setting.

\begin{proof}[Proof (Theorem \ref{lem:trivial-action}, result (i))]
    Let $X$ be a simply-connected space. We need to show that $(p'\,.\, u =_{\pi_n(X,x_0)} u)$ for all $p': \pi_1(X,x_0)$ and $u : \pi_n(X,x_0)$\,; this is a proposition, so it suffices to show this claim for truncated loops $p' \jdeq \tproj 0 p$ only. Let $p : x_0 =_X x_0$\,, then since~$X$ is simply-connected, there exists a mere homotopy $\brck{p^{-1} = \refl{x_0}}$ between $p^{-1}$ and the constant path $\refl{x_0}$\,. Again, we can strip the truncation to obtain an explicit homotopy~$h : p^{-1} = \refl{x_0}$\,, which implies that
    \[
        \tproj 0 p \,.\, u \jdeq \transfib{\pi_n(X,\blank)}{p^{-1}}{u} = \transfib{\pi_n(X,\blank)}{\refl{x_0}}{u} \jdeq u~,
    \]
    for all $u : \pi_n(X,x_0)$\,.
\end{proof}

\begin{proof}[Proof (Theorem \ref{lem:trivial-action}, result (ii))]
    First, note that since truncation $\tprojf 0 : (x = x) \to \pi_1(X,x)$ commutes with transport (Lemma \ref{lem:comm_with_transport}), and transport in the loop space $(x_0 = x_0)$ equals conjugation (Lemma 2.11.2 in \cite{hottbook}), we have that for all loops $p, q : x_0=_X x_0$\,,
    \[
        \transfib{\pi_1(X,\blank)}{p^{-1}}{\tproj{0}{q}} = \tproj{0}{\transfib{x \mapsto x =_X x}{p^{-1}}{q}} = \tproj{0}{p \ct q \ct p^{-1}} \jdeq \tproj{0}{p} \cdot \tproj{0}{q} \cdot \tproj{0}{p}^{-1}~,
    \]
    where the final equality is because of the definition of the group operations on $\pi_1(X,x)$\,.
    
    Now, assume that $\pi_1(X,x_0)$ is abelian and let $p', q' :\pi_1(X,x_0)$\,. Since the goal is to show a proposition, we may assume that $p'\jdeq \tproj{0}{p}$ and $q' \jdeq \tproj{0}{q}$ for some loops $p, q : x_0 =_X x_0$\,, so by commutativity of $\pi_1(X,x_0)$ we have that
    \[
        \transfib{\pi_1(X,\blank)}{p^{-1}}{\tproj{0}{q}} = \tproj{0}{p} \cdot \tproj{0}{q} \cdot \tproj{0}{p}^{-1} = \tproj{0}{q}~.
    \]

    Conversely, assume that the $\pi_1(X,x_0)$-action on itself is trivial and let $p', q' :\pi_1(X,x_0)$\,. Since the goal is to show a proposition, we may again assume that $p'\jdeq \tproj{0}{p}$ and $q' \jdeq \tproj{0}{q}$ for some loops $p, q : x_0 =_X x_0$\,. Then, using the same equation as before, it holds that
    \[
        \tproj{0}{p} \cdot \tproj{0}{q} \cdot \tproj{0}{p}^{-1} = \transfib{\pi_1(X,\blank)}{p^{-1}}{\tproj{0}{q}} = \tproj{0}{q}~,
    \]
    so $p' \cdot q' \jdeq \tproj{0}{p} \cdot \tproj{0}{q} = \tproj{0}{q} \cdot \tproj{0}{p} \jdeq q' \cdot p'$\,.
\end{proof}

Before proving result (iii) of Theorem \ref{lem:trivial-action}, let us discuss its assumption, namely that
\begin{equation}
\text{\textit{``merely for all loops $p, q:\Omega(X,x_0)\,$ it holds that $p \ct q = q \ct p$''}.}\label{eq:hyp_res_3}
\end{equation}
On the surface, this seems to just say that the fundamental group $\pi_1(X,x_0)$ is abelian, but statement (\ref{eq:hyp_res_3}) is stronger. Consider, for example, the space $S^1 \vee S^2$\,. Its fundamental group is abelian, but the action of $\pi_1$ on $\pi_2$ is not trivial. If (\ref{eq:hyp_res_3}) could be weakened to just demanding that the fundamental group $\pi_1(X,x_0)$ is abelian, then $S^1 \vee S^2$ would serve as a counterexample to result (iii).

The subtle difference between these statements is caused by where the propositional truncation, indicated by the word \emph{merely}, is placed. In HoTT, there exists a map 
\[
     \Brck{\prd{p,q : \Omega(X,x_0)} p \ct q = q \ct p\,} ~~\longrightarrow~ \prd{p,q : \Omega(X,x_0)} \Brck{\,p \ct q = q \ct p\,}~,
\]
but, in general, this is not an equivalence. The type on the left-hand side expresses statement~(\ref{eq:hyp_res_3}), the type on the right-hand side is equivalent, via Theorem 7.3.12 in \cite{hottbook}, to the statement that the fundamental group $\pi_1(X,x_0)$ is abelian. According to Remark~3.8.4 in the HoTT book \cite{hottbook}, it is admissible to assume that an inverse exists --- which is equivalent to assuming the axiom of choice --- if the loop space $\Omega(X,x_0)$ is a set. In general this is not the case, as with the space $S^1 \vee S^2$\,.

What are the kind of spaces that satisfy the hypothesis in result (iii)? Classically, loop spaces in which composition is commutative are called \emph{homotopy commutative}. The commutativity is allowed to only hold up to homotopy, which is the default setting for equality between paths in homotopy type theory. By the Eckmann-Hilton argument, every H-space has a commutative loop space, but not every commutative loop space arises in this way, e.g. James Stasheff \cite[Thm.\,1.18]{stasheff_1961} proved that $\mathbb{CP}^3$ has a commutative loop space. Hideyuki Kachi \cite{kachi} also presents a plethora of finite CW complexes with commutative loop spaces, their focus is on simply-connected spaces --- for which triviality of the $\pi_1$-action is already covered by result (i) --- but they also briefly discuss non simply-connected CW~complexes.

Proving result (iii) takes more work than results (i) and (ii).

\begin{proof}[Proof (Theorem \ref{lem:trivial-action}, result (iii))]
    Assume that merely for all loops $p,q  : \Omega(X,x_0)$ it holds that $p \ct q = q \ct p$\,. We claim that it suffices to merely show that transport in $\Omega^n$ is trivial, i.e. that the type 
    \begin{equation}
        \Brck{~\prod_{p : \Omega(X,x_0)} \prod_{u : \Omega^n(X,x_0)} \transfib{\Omega^n(X,\blank)}{p^{-1}}{u} =_{\Omega^n(X,x_0)} u~} \label{eq:r3_new_goal}
    \end{equation}
    is inhabited. This type implies that for all loops $p : x_0 =_X x_0$ and $u : \Omega^n(X,x_0)$ the mere path type $\brck{\transfib{\Omega^n(X,\blank)}{p^{-1}}{u} =_{\Omega^n(X,x_0)} u}$ is inhabited, which by Theorem 7.3.12 in \cite{hottbook}, is equivalent to an equality between truncated loops, $\mathopen{}| \transfib{\Omega^n(X,\blank)}{p^{-1}} u |_{0} \mathclose{} =_{\pi_n(X,x_0)} \tproj 0 u$\,. 
    As $\tprojf 0 : \Omega^n(X,x) \to \pi_n(X,x)$ commutes with transport by Lemma~\ref{lem:comm_with_transport}, this implies that
    \[
        \tproj 0 p \,.\, \tproj 0 u \jdeq \transfib{\pi_n(X,
        \blank)}{p^{-1}}{\tproj 0 u} = \mathopen{}| \transfib{\Omega^n(X,\blank)}{p^{-1}} u |_{0} \mathclose{} = \tproj 0 u~,
    \]
    which says that the action of $\pi_1(X,x_0)$ on $\pi_n(X,x_0)$ is trivial. (Since triviality of the $\pi_1(X,x_0)$-action is a proposition, it suffices to only show the triviality for truncated loops of the form $\tproj 0 p : \pi_1(X,x_0)$ and $\tproj 0 u : \pi_n(X,x_0)$\,.)

    We proceed to show statement (\ref{eq:r3_new_goal}) by induction on $n$\,. The base case $n = 1$ follows directly from the assumption. Since (\ref{eq:r3_new_goal}) is a proposition, the `merely' part in the initial assumption can be stripped, leaving us with the fact that composition in the loop space is commutative. Together with Lemma 2.11.2 in \cite{hottbook}, which says that transport in $\Omega(X,x) \jdeq (x =_X x)$ equals conjugation, it holds that 
    \[
        \transfib{\Omega(X,\blank)}{p^{-1}}{q} = p \ct q \ct p^{-1} = q~~~~~~~\text{for all loops }p,q : (x_0 =_X x_0)~.
    \]

    For the inductive step, assume that (\ref{eq:r3_new_goal}) holds for $n$, we need to show it holds for $n+1$\,. Since both the induction hypothesis and the goal are propositionally truncated, we can strip both truncations. The resulting induction hypothesis says that transport in $\Omega^n$ is trivial, i.e.
    \[
        \transfib{\Omega^n(X,\blank)}{p^{-1}}{u} =_{\Omega^n(X,x_0)} u~~~~~~~\text{for all }p : x_0 =_X x_0 \text{ and }u : \Omega^n(X,x_0)~.
    \]
    By functional extensionality, this gives an equality between maps $\Omega^n(X,x_0) \to \Omega^n(X,x_0)$\,,     \begin{align}
        h_p : \transfib{\Omega^n(X,\blank)}{p^{-1}}{\blank} = \idfunc[\Omega^n(X,x_0)]~. \label{eq:h}
    \end{align}
    Now, let $p : x_0 =_X x_0$ and $v : \Omega^{n+1}(X,x_0)$\,. The goal is to show that transport in $\Omega^{n+1}$ is trivial, i.e.
    \[
        \transfib{\Omega^{n+1}(X,\blank)}{p^{-1}}{v} =_{\Omega^{n+1}(X,x_0)} v~.
    \]

    The next part involves a lot of complicated equations. We use a theorem from the HoTT book \cite{hottbook} to relate transport in $\Omega^{n+1}$ to transport in $\Omega^n$\,; from there we use the induction hypothesis, namely the path $h_p$\,, to relate transport in $\Omega^n$ to the identity on $\Omega^n$\,; using the previous two steps, we obtain an expression of transport in $\Omega^{n+1}$ as a conjugation of the original $(n+1)$-dimensional loop $v$ in $\Omega^{n+1}$\,; using that composition in $\Omega^{n+1}$ commutes, we get that transport in $\Omega^{n+1}$ is trivial.
    
    By Theorem 2.11.4 \cite{hottbook}, transport in $\Omega^{n+1}$ can be related to transport in $\Omega^n$\,. Recall that $v : \Omega^{n+1}(X,x_0) \jdeq (\rrefl {x_0} n = \rrefl{x_0}n)$\,, so
    \begin{align}
        &\transfib{\Omega^{n+1}(X,\blank)}{p^{-1}}{v} \nonumber \\
        &~~~~~~~~= (\apdfunc{\rrefl{(\blank)}{n}}(p^{-1}))^{-1} \ct (\apfunc{\transfib{\Omega^n(X,\blank)}{p^{-1}}{\blank}}(v)) \ct (\apdfunc{\rrefl{(\blank)}{n}}(p^{-1}))~, \label{eq:transport_lower_loop_space}
    \end{align}
    where $\apdfunc{\rrefl{(\blank)}{n}}(p^{-1}) : \transfib{\Omega^n(X,\blank)}{p^{-1}}{\rrefl{x_0}n} =_{\Omega^n(X,x_0)} \rrefl{x_0}n$ is a dependent equality between $\rrefl{x_0}n$ and itself.

    Using the path $h_p$ from the induction hypothesis (\ref{eq:h}), the middle term in (\ref{eq:transport_lower_loop_space}) can be rewritten as
    \begin{align}
        &\apfunc{\transfib{\Omega^n(X,\blank)}{p^{-1}}{\blank}}(v)
            = \transfib{\varphi \mapsto \varphi(\rrefl{x_0}{n}) = \varphi(\rrefl{x_0}{n})}{h_p^{-1}}{\apfunc{\idfunc[\Omega^n]}(v)} \nonumber \\
        &~~~~~~~~= (\apfunc{\varphi \mapsto \varphi(\rrefl {x_0} {n})}{(h_p^{-1})})^{-1} \ct \apfunc{\idfunc[\Omega^n]}(v) \ct (\apfunc{\varphi \mapsto \varphi(\rrefl {x_0} {n})}{(h_p^{-1})}) \nonumber \\
        &~~~~~~~~\jdeq (\apfunc{\varphi \mapsto \varphi(\rrefl {x_0} {n})}{(h_p^{-1})})^{-1} \ct v \ct (\apfunc{\varphi \mapsto \varphi(\rrefl {x_0} {n})}{(h_p^{-1})})~, \label{eq:middle_eq_id}
    \end{align}
    where the second equality follows from Theorem 2.11.3 in \cite{hottbook}. The term $\apfunc{\varphi \mapsto \varphi(\rrefl {x_0} {n})}{(h_p^{-1})}$ is of type $\rrefl{x_0}n \jdeq \idfunc[\Omega^n](\rrefl{x_0}n) = \transfib{\Omega^n(X,\blank)}{p^{-1}}{\rrefl{x_0}{n}}$\,, the reverse equality of the term~$\apdfunc{\rrefl{(\blank)}{n}}(p^{-1})$ in (\ref{eq:transport_lower_loop_space}).

    Combining equations (\ref{eq:transport_lower_loop_space}) and (\ref{eq:middle_eq_id}), we have that $\transfib{\Omega^{n+1}(X,\blank)}{p^{-1}}{v}$ equals
    \begin{align*}
        (\apdfunc{\rrefl{(\blank)}{n}}(p^{-1}))^{-1} \ct (\apfunc{\varphi \mapsto \varphi(\rrefl {x_0} {n})}{(h_p^{-1})})^{-1} \ct v \ct (\apfunc{\varphi \mapsto \varphi(\rrefl {x_0} {n})}{(h_p^{-1})}) \ct (\apdfunc{\rrefl{(\blank)}{n}}(p^{-1}))~. 
    \end{align*}
    Note that $(\apfunc{\varphi \mapsto \varphi(\rrefl {x_0} {n})}{(h_p^{-1})}) \ct (\apdfunc{\rrefl{(\blank)}{n}}(p^{-1}))$ is a $(n+1)$-dimensional loop of the type $\rrefl{x_0}n = \rrefl{x_0}{n}$\,, just like $v$\,. Transportation in $\Omega^{n+1}$ thus comes down to conjugation with some $(n+1)$-cell, just like in $\Omega^1$\,. Since $n+1 \geq 2$\,, composition of $(n+1)$-cells in $\Omega^{n+1}$ is commutative, which gives us that
    \[
        \transfib{\Omega^{n+1}(X,\blank)}{p^{-1}}v = \widetilde{p}^{-1} \ct v \ct \widetilde{p} = v~,
    \]
    with $\widetilde{p} \defeq (\apfunc{\varphi \mapsto \varphi(\rrefl {x_0} {n})}{(h_p^{-1})}) \ct (\apdfunc{\rrefl{(\blank)}{n}}(p^{-1}))$\,, i.e. transport in $\Omega^{n+1}$ is trivial.
\end{proof}

\subsection{Relation to free pointedness}
Whereas we used change-of-basepoint isomorphisms $\pi_n(S^n,x) \cong \pi_n(S^n,y)$ to define the degree for non-pointed maps $S^n \to S^n$\,, Hou \cite{favonia-thesis} takes a different approach: they observe that the degree-map $(S^n \to S^n) \to \mathbb{Z}$ takes values in a set, so it suffices to define the map on the set-truncation $\trunc{0}{S^n \to S^n}$\,. They then use that, on the level of sets, pointedness is free, meaning that for $n \geq 1$ the projection that forgets about pointedness is an equivalence, i.e.
\begin{equation*}
    \trunc{0}{(S^n,\base) \pto (S^n, \base)} ~\eqvsym~ \trunc{0}{S^n \to S^n}~.
\end{equation*}
In \cite{favonia-cellcohom}, Buchholtz and Hou provide an intuition as to why this map is an equivalence: for any map $f : S^n \to S^n$\,, the suspension $\mathsf{susp}(f) : S^{n+1} \to S^{n+1}$ is automatically pointed and by the Freudenthal suspension theorem all maps $S^{n+1} \to S^{n+1}$ are of this form.

Hou's approach to defining the degree map and the approach in this article are more similar than they appear on first sight. In this section, we show a classical statement relating basepoint-preserving and free homotopy classes of functions. In HoTT, these are encoded as the sets of pointed maps $\trunc{0}{(Z,z_0) \pto (X, x_0)}$ and non-pointed maps $\trunc{0}{Z \to X}$\,. From this classical statement, it follows that the forgetful map
\[
    \trunc{0}{(Z,z_0) \pto (X, x_0)} ~\longrightarrow~ \trunc{0}{Z \to X}
\]
is an equivalence precisely when the $\pi_1(X,x_0)$-action on pointed maps $\trunc{0}{(Z,z_0) \pto (X, x_0)}$ is trivial. Moreover, for $(Z,z_0) \defeq (S^n,\base)$\,, the $\pi_1$-action on $\trunc{0}{(S^n,\base) \pto (X, x_0)}$ coincides with the $\pi_1$-action on higher homotopy groups $\pi_n$ under the equivalence
\[
    \trunc{0}{(S^n,\base)) \pto (X, x_0)}~\cong~\pi_n(X,x_0)
\]
from Lemma 6.5.4 in the HoTT book \cite{hottbook}. Thus, in the end both approaches rest on the triviality of the $\pi_1$-action on the higher homotopy groups of spheres.

Similar to the $\pi_1$-action on higher homotopy groups (Definition \ref{def:pi1-action}), the $\pi_1$-action on the set of pointed maps is defined using transport. Following Hatcher \cite{hatcherAT}, this action is defined as a right action, as opposed to the action on higher homotopy groups which is a left action. 

\begin{definition}[$\pi_1$-action on pointed maps]
    Let $(Z,z_0)$ and $(X,x_0)$ be pointed types. The action of $\pi_1(X,x_0)$ on the set of pointed maps $\trunc{0}{(Z,z_0) \pto (X,x_0)}$ is defined on truncated loops $\tproj 0 p : \pi_1(X,x_0)$ as
    \[
        f \,.\, \tproj 0 p \defeq \transfib{\trunc{0}{(Z,z_0) \pto (X,\blank)}}{p}{f}
    \]
    with $f : \trunc{0}{(Z,z_0) \pto (X,x_0)}$\,.
\end{definition}

Alternatively, the $\pi_1$-action on pointed maps can be expressed as follows.

\begin{lemma} \label{lem:alt_transp_pmaps}
    Let $(Z,z_0)$ and $(X,x_0)$ be pointed types. For all loops $p : x_0 =_X x_0$ and truncated, pointed maps $\tproj 0 {(f,w_f)} : \trunc{0}{(Z,z_0) \pto (X,x_0)}$\,, it holds that
    \[
        \tproj 0 {(f,w_f)}\,.\, \tproj 0 p = \tproj 0 {(f,w_f \ct p)}~, 
    \]
    i.e. the $\pi_1$-action only acts on the proof of pointedness $w_f : f(z_0) =_X x_0$\,.
\end{lemma}

\begin{proof}
    Let $\tproj 0 p : \pi_1(X,x_0)$ and $\tproj 0 {(f,w_f)} : \trunc{0}{(Z,z_0) \pto (X,x_0)}$\,. Transport in the type family $((Z,z_0) \pto (X,\blank)) \jdeq \sm{g : Z \to X} g(z_0) =_X (\blank)$  can be rewritten as
    \begin{equation}
        \transfib{(Z,z_0) \pto (X,\blank)} p {(f,w_f)} = (f, \transfib{f(z_0) = (\blank)} p {w_f})~. \label{eq:rewrite-pi1-action-1}
    \end{equation}
    This follows from a more general statement where $p$ is a free path (meaning that we also have universal quantification over its endpoints), which holds by path induction. Using equation~(\ref{eq:rewrite-pi1-action-1}) and that truncation $\tprojf 0$ commutes with transport by Lemma \ref{lem:comm_with_transport}, it holds that
\begin{align*}
    & \transfib{\trunc{0}{(Z,z_0) \pto (X,\blank)}}{p}{\tproj{0}{(f,w_f)}} = \mathopen{}|\transfib{(Z,z_0) \pto (X,\blank)}{p}{(f,w_f)}|_0\mathclose{} \\
    &~~~~~~= \mathopen{}|(f, \transfib{f(z_0) = (\blank)}{p}{w_f})|_0\mathclose{} = \tproj{0}{(f,w_f \ct p)}~,
\end{align*}
where the final equality is by Theorem 2.11.2 in \cite{hottbook}.
\end{proof}

The $\pi_1$-action on pointed maps coincides with the $\pi_1$-action on higher homotopy groups, modulo an inversion to account for the difference in left and right action.
\begin{lemma}
    Let $(X,x_0)$ be a pointed type, then there exists an equivalence
    \[
        \varphi : \trunc{0}{(S^n,\base) \pto (X,x_0)} ~\cong~ \pi_n(X,x_0)~,
    \]
    and \,$\varphi(f \,.\, p) = p^{-1} \,.\, \varphi(f)$ for loops $p : \pi_1(X,x_0)$ and maps $f : \trunc{0}{(S^n,\base) \pto (X,x_0)}$\,.
\end{lemma}

\begin{proof}
    The equivalence $\varphi$ comes from Lemma 6.5.4 in \cite{hottbook}\,. Moreover, we actually have a family of equivalences $\varphi_x : \trunc{0}{(S^n,\base) \pto (X,x)} ~\cong~ \pi_n(X,x)$\,. Let $p' : \pi_1(X,x_0)$ and $f : \trunc{0}{(S^n,\base) \pto (X,x_0)}$\,. The goal is to show a proposition, namely an equality in the set $\pi_n(X,x_0)$\,, so we can assume that $p' \jdeq \tproj 0 p$ for some $p : x_0 =_X x_0$\,. By Lemma \ref{lem:comm_with_transport}, the family of equivalences commutes with transport, which gives us that
    \[
         \varphi(f \,.\, \tproj 0 p) \jdeq \varphi(\transfib{\trunc{0}{(S^n,\base) \pto (X,\blank)}}{p}{f}) = \transfib{\pi_n(X,\blank)}{p}{\varphi(f)} = \tproj 0 {p}^{-1} \,.\, \varphi(f)~.
    \]
\end{proof}

We now show the classical statement relating basepoint-preserving homotopy classes $\trunc 0 {(Z,z_0) \pto (X,x_0)}$ and free homotopy classes $\trunc 0 {Z \to X}$ in homotopy type theory. Note that the classical formulation requires $Z$ to be restricted to the class of CW complexes, but this is not necessary in HoTT since $\infty$-groupoids already behave well enough.
\begin{lemma}[cf.~{\cite[Prop.\,4A.2]{hatcherAT}}]
    Let $(Z,z_0)$ and $(X,x_0)$ be pointed types with $X$ connected. The map that forgets pointedness induces an equivalence between the orbit set $\trunc{0}{(Z,z_0) \pto (X,x_0)} \,/\, \pi_1(X,x_0)$ and $\trunc{0}{Z \to X}$\,.
\end{lemma}

The proof is kind of similar to the classical proof, but one needs to be able to read through the homotopical interpretation to see the correspondence.

\begin{proof}
    First, note that $\fst : ((Z,z_0) \pto (X,x_0)) \to (Z \to X)$ indeed induces a map
    \[
        \psi :\,~\trunc{0}{(Z,z_0) \pto (X,x_0)} \,/\, \pi_1(X,x_0)~~ \longrightarrow~~ \trunc{0}{Z \to X}~,
    \]
    which is given by $\psi ([\,\tproj 0 {(f,w_f)}]) \defeq \tproj 0 f$~.
    To show that $\psi$ is well-defined, we need to prove that for all maps $f_1', f_2' : \trunc{0}{(Z,z_0) \pto (X,x_0)}$ in the same orbit it holds that $\psi(f_1') = \psi(f_2')$\,. (The higher coherence conditions are automatically satisfied because $\trunc 0 {Z \to X}$ is a set.) If $f_1'$ and $f_2'$ are in the same orbit, this means there exists a truncated loop $p' : \pi_1(X,x_0)$ and a path $w_{p'} : (f_1' \,.\, p') = f_2'$\,. (Technically, we have only mere existence of $p'$ and $w_{p'}$\,, but this can be stripped since $\psi(f_1') = \psi(f_2')$ is a proposition). Since one endpoint of $w_p$ is free, namely $f_2'$\,, we can perform path induction on $w_{p'}$\,, meaning that we may assume $f_2' \jdeq (f_1' \,.\, p')$\,. The goal is then to show that $\psi(f_1') = \psi(f_1'\,.\,p')$\,. We may also assume that $p' \jdeq \tproj 0 p$ and $f_1' \jdeq \tproj 0 {(f_1, w_{f_1})}$ for some $p : x_0 =_X x_0$ and $(f_1, w_{f_1}) : (Z,z_0) \pto (X,x_0)$\,.
    By the alternative expression for the $\pi_1(X,x_0)$-action (Lemma \ref{lem:alt_transp_pmaps}), it holds that
    \[
        f_1' \,.\, p' \jdeq \tproj 0 {p} \,.\, \tproj 0 {(f_1, w_{f_1})} = \tproj 0 {(f_1, w_{f_1} \ct p)}~,
    \]
    from which it follows that
    \[
        \psi(f_1') \jdeq \psi(\tproj 0 {(f_1, w_{f_1})}) \jdeq \tproj 0 {f_1} \jdeq \psi (\tproj 0 {(f_1, w_{f_1} \ct p)}) = \psi(f_1' \,.\, p')~.
    \]

    Next, we show that $\psi$ is surjective. Let $f : Z \to X$\,, the goal is to \emph{merely} construct a term $f_\bullet : \trunc 0 {(Z,z_0) \pto (X,x_0)} \,/\, \pi_1(X,x_0)$ such that $\psi(f_\bullet) = \tproj 0 f$\,. (We may strip the truncation from $\trunc 0 {Z \to X}$ because the goal is a proposition). Since $X$ is connected, $\brck{f(z_0) =_X x_0}$ is inhabited. Note that the goal is to merely construct a term in the preimage of $f$\,, so we can strip this truncation, which yields an explicit path $p : f(z_0) =_X x_0$\,. We use this path to define $f_\bullet$\,, namely
    \[
        f_\bullet \defeq [\,\tproj 0 {(f,p)}] ~\,:~\,\trunc 0 {(Z,z_0) \pto (X,x_0)} \,/\, \pi_1(X,x_0)~,
    \]
    and we have that $\psi(f_\bullet) \jdeq \psi([\,\tproj 0 {(f,p)}]) \jdeq \tproj 0 f$\,.

    To show injectivity, consider two terms in $\trunc{0}{(Z,z_0) \pto (X,x_0)} \,/\, \pi_1(X,x_0)$ of the form $[\,\tproj{0}{(f_1,w_{f_1})}]$ and $[\,\tproj{0}{(f_2,w_{f_2})}]$ which satisfy $\psi([\,\tproj{0}{(f_1,w_{f_1})}]) = \psi([\,\tproj{0}{(f_2,w_{f_2})}])$\,, meaning $\tproj 0 {f_1} = \tproj 0 {f_2}$\,. The goal is to show that $[\,\tproj{0}{(f_1,w_{f_1})}] = [\,\tproj{0}{(f_2,w_{f_2})}]$\,.
    (We may assume the terms $[\,\tproj{0}{(f_i,w_{f_i})}]$ to be of this form because the goal is a proposition, which not only allows us to strip the truncation from $\trunc{0}{(Z,z_0) \pto (X,x_0)}$\,, but also to consider specific representatives of the equivalence classes.) By Theorem 7.3.12 in \cite{hottbook}, the assumption $\tproj 0 {f_1} = \tproj 0 {f_2}$ is equivalent to a mere equality $\brck{f_1 = f_2}$\,. Stripping the truncation gives an explicit path $h' : f_1 = f_2$\,, and thus also a homotopy $h : \prod_{z} f_1(z) = f_2(z)$\,.
    The goal is equivalent to showing that $\tproj{0}{(f_1,w_{f_1})}$ and $\tproj{0}{(f_2,w_{f_2})}$ belong to the same orbit, i.e. showing there (merely) exists a loop $p : \pi_1(X,x_0)$ such that $\tproj{0}{(f_1,w_{f_1})} \,.\, p =\tproj{0}{(f_2,w_{f_2})}$\,. Choose $p \defeq | w_{f_1}^{-1} \ct h(z_0) \ct w_{f_2} |_0$\,, then by Lemma \ref{lem:alt_transp_pmaps} it holds that
    \[
        \tproj{0}{(f_1,w_{f_1})} \,.\, |w_{f_1}^{-1} \ct h(z_0) \ct w_{f_2}|_0 = | (f_1, w_{f_1} \ct (w_{f_1}^{-1} \ct h(z_0) \ct w_{f_2})) |_0= \tproj{0}{(f_1,h(z_0) \ct w_{f_2})}~.
    \]
    It thus suffices to show that $(f_1, h(z_0)\ct w_{f_2}) = (f_2,w_{f_2})$\,. To turn the equality $h'~:~f_1~=~f_2$ into an equality of pointed maps $(f_1, h(z_0)\ct w_{f_2}) = (f_2,w_{f_2})$\,, we need to show that $(\transfib{f \mapsto f(z_0) = x_0}{h'}{h(z_0) \ct w_{f_2}} = w_{f_2})$\,, which holds since
    \[
        \transfib{f \mapsto f(z_0) = x_0}{h'}{h(z_0) \ct w_{f_2}} = h(z_0)^{-1} \ct (h(z_0) \ct w_{f_2}) = w_{f_2}~,
    \]
    the first equality follows from path induction on $h'$\,. This concludes the proof that $\psi$ is a well-defined equivalence.
\end{proof}

By the previous lemma, the forgetful map $\trunc{0}{(Z,z_0) \pto (X, x_0)} ~\longrightarrow~ \trunc{0}{Z \to X}$
is an equivalence precisely when the $\pi_1$-action on pointed maps is trivial. Hence, both the free pointedness of maps $\trunc{0}{(Z,z_0) \pto (X, x_0)}$\,, and the existence of canonical change-of-basepoint isomorphisms $\pi_n(X,x) \cong \pi_n(X,y)$\,, are consequences of trivial $\pi_1$-actions, the same $\pi_1$-action, in fact, when $(Z,z_0) \defeq (S^n,\base)$\,.

\section{Discussion and Conclusion}
In this article, we developed parts of algebraic topology in the synthetic language provided by homotopy type theory (HoTT). We proved a synthetic version of the classification of covering spaces, and synthetically explored the existence of canonical change-of-basepoint isomorphisms between homotopy groups. 

Developing this theory synthetically required translating the classical definitions and statements into HoTT and there is some freedom in picking what translation to use. Some translations are easier to work with than others.
For those interested in formalizing some topics from algebraic topology in HoTT themselves, we recommend the following:
\begin{enumerate}
    \item Use type families $P : X \to \ttype$ instead of the total space $\sum_X P$\,. Although this perspective takes some getting used to (especially when coming from classical mathematics), using type families directly allows more relations to be encoded as judgmental equalities instead of propositional ones, saving you from having to carry these around using transport.
    \item Some direct translations using concepts from the HoTT book \cite{hottbook} can be simplified. In the case of the lifting criterion, we found an alternative condition that saved us from a lot of truncations (see Lemma \ref{lem:eqv-lift-crit}).
\end{enumerate}

We were able to closely mirror the classical proofs found in a standard reference like Hatcher's \emph{Algebraic Topology} \cite{hatcherAT}. Extension by weak constancy (Lemma \ref{lem:ext_weak_constancy}) proved vital to reproduce classical constructions that rely on the mere existence of paths. The transport operation does feature more prominently in HoTT proofs than in the classical ones: it is used to construct both homotopy lifts, homotopy extensions, and to define dependent paths. The notions of transport and dependent paths do faithfully represent the geometrical reasoning in the classical proofs. (The abundance of transport operations did make us appreciate all the lemmas in the HoTT book characterizing transport in different type families \cite[Ch. 2]{hottbook}.)

We see both advantages and disadvantages with the synthetic approach to algebraic topology.
The low-level encoding of mathematical concepts allows one to be more explicit about the objects involved without being overwhelmed by notation, and the theory can be mechanized in proof assistants much more easily. 
Truncations are a nice way to make explicit the difference between constructive and platonic existence statements in mathematics, as with the mere existence of a paths in a path-connected spaces. On the other hand, they are a nuisance to deal with. In Section \ref{sec:classification}, we managed to avoid a lot of truncations since the fibers~$F(x)$ of a covering space are sets, but we were not so lucky in Section \ref{sec:change_of_basepoint}. By the end we had gotten used to working with them --- and hopefully the reader with us.

At the start of this project, we only had a passing knowledge of what made HoTT special compared to the standard flavor of Martin-L\"of type theory. Having formalized some classical results from algebraic topology ourselves, we have gained a better understanding and deeper appreciation for concepts like transport, truncations, and univalence.



\bibliography{general_biblio}

\end{document}